\documentclass[a4paper,12pt]{article}
\RequirePackage[T1]{fontenc}
\RequirePackage{fix-cm}
\usepackage{amsmath,amsfonts,amsthm}
\usepackage{graphicx} 
\usepackage[unicode,final]{hyperref} 
\usepackage{cite}
\usepackage{caption}
\usepackage{mathtools}
\usepackage{newtxtext}
\usepackage[varvw]{newtxmath}
\usepackage{bm}
\usepackage[T1]{fontenc}
\usepackage[utf8]{inputenc}
\usepackage{xcolor}
\usepackage{booktabs}
\usepackage{mathtools}
\usepackage{tikz, pgfplots}
\pgfplotsset{compat=1.17}
\usepgfplotslibrary{groupplots}
\usetikzlibrary{shapes.geometric}
\usetikzlibrary{backgrounds}
\usetikzlibrary{intersections}

\newcommand{\ddt}{\partial_t}
\newcommand{\ddx}{\partial_x}

\newtheorem{Def}{Definition}

\newtheorem{Prop}{Proposition}

\newtheorem{Cor}{Corollary}
\newtheorem{Rmk}{Remark}

\allowdisplaybreaks[4]
\newcommand{\R}{\mathbb{R}}
\newcommand{\N}{\mathbb{N}}
\newcommand{\Z}{\mathbb{Z}}

\newcommand{\eps}{\varepsilon}
\providecommand{\keywords}[1]{\textit{Keywords:} #1}

\makeatletter
\renewenvironment{proof}[1][\proofname]{\par
  \normalfont%
  \topsep6\p@\@plus6\p@ \trivlist%
  \item[\hskip\labelsep{\itshape#1}\@addpunct{\itshape.}]\ignorespaces
  }{\qed\endtrivlist}
\renewcommand{\proofname}{Proof}
\makeatother

\usepackage{soul} 
\begin{document}
\title{Central schemes for networked scalar conservation laws}
\author{
  Michael~Herty$^{1}$ \and
  Niklas~Kolbe$^{1}$\footnote{Corresponding author} \and
  Siegfried~Müller$^{1}$
}
%
\date{
 \small
  $^1$Institute of Geometry and Practical Mathematics, RWTH Aachen University,\\ Templergraben 55, 52062 Aachen, Germany\\
   \smallskip
   {\tt \{herty,kolbe,mueller\}@igpm.rwth-aachen.de} \\
   \smallskip
   \today
 }
\maketitle
\begin{abstract}
We propose a novel scheme to numerically solve scalar conservation laws on networks without the necessity to solve Riemann problems at the junction. The scheme is derived using the relaxation system introduced in [Jin and Xin, Comm. Pure Appl. Math. 48(3), 235–276 (1995)] and taking the relaxation limit also at the nodes of the network. The scheme is mass conservative and yields well defined and easy-to-compute coupling conditions even for general networks. We discuss higher order extension of the scheme and applications to traffic flow and two-phase flow. In the former we compare with results obtained in literature.
\par
\vspace{0.5em}
\noindent

\end{abstract}
%
\section{Introduction}\label{sec:intro}
Research on mathematical models on networks understood as directed (one-dimensional) graphs has been successfully conducted over the last decades and we refer to the recent survey \cite{MR3200227} for details and references. Such models have various applications, such as gas dynamics in pipelines~\cite{k.bandaCouplingConditionsGas2006,MR2818413,MR4039520}, vehicular traffic on road networks~\cite{garavelloTrafficFlowNetworks2006,MR1338371}, production systems \cite{MR2665143} and blood flow through systems of blood vessels~\cite{formaggiaMultiscaleModellingCirculatory1999} to mention only a few. Apart from modeling questions regarding the partial differential equations on the edges, a major modeling challenge is the description of the dynamics at the network nodes, where adjacent edges connect. Starting with \cite{MR1338371,k.bandaCouplingConditionsGas2006} networked (systems of) conservation or balance laws are defined by  (physically induced) coupling conditions, see \cite{MR3200227} for examples, and e.g  \cite{MR4175145} for a hierarchical derivation in the case of gas dynamics. Those conditions yield under suitable assumptions boundary conditions. Using wave--front--tracking techniques well--posedness of such coupled problems could be established, see e.g.  \cite{MR1338371,garavelloTrafficFlowNetworks2006,MR2438778,MR2377285,MR2735916}. A key analytical concept here is the notion of Riemann solvers \cite{garavelloTrafficFlowNetworks2006,MR3553143} or half--Riemann problems \cite{MR2237163}. Among others, those concepts require an analytic expression of wave curves of the corresponding models. Here, we are interested in the numerical treatment of networked conservation laws. Numerical methods based on the Riemann solvers have been proposed already in \cite{MR2045460,MR1338371,MR2223073} and have recently gained interest in view of high--order methods \cite{MR3396266,Banda2016,MR2198203,MR3227276,MR2963941,MR3315275,MR3528308},  property--preserving schemes \cite{MR4026004,MR4260434} and also for problems where Lax--curves and eigenvalues are not explicit \cite{MuellerVoss:2006,MR2600931} or not available \cite{HantkeMueller:2018,HantkeMueller:2019}. Regarding the development of  efficient schemes, most of the higher--order schemes rely on the linearization of the coupling conditions such that the Lax--curves are obtained trivially, see e.g. \cite{Banda2016,MR3227276} for more details. Linearization techniques have also been used  to avoid the problem of the explicit computations of eigenvalues and Lax-curves in two--phase problems, see e.g. \cite{MR3265942}.  
The question of property preserving numerical schemes across networks has been recently investigated in view of well-balanced networked equations \cite{MR4026004} and entropy-preserving schemes \cite{MR4260434}.  Note that for simplicity we consider here only the case of conservation laws but the numerical schemes directly extend to the case of balance laws. Further, we focus here on finite-volume or discontinuous Galerkin schemes whereas approaches based on finite-element schemes, such as \cite{MR3744998}, require a different treatment of coupling conditions. Similarly, a construction of vanishing viscosity solutions, which has been addressed in the schemes studied in \cite{karlsenConvergenceGodunovScheme2017, towersExplicitFiniteVolume2022}, avoids the use of Riemann problems at the expense of formally treating parabolic systems. 
\par 
In this manuscript we develop a numerical scheme that does not require a Riemann solver at the junction. To this end, we embed the coupling problem of nonlinear equations in the coupling of linear relaxation systems following the approach developed by e.g. Jin-Xin  cf.~\cite{jinRelaxationScheme1995}. Since for linear systems the Lax curves are multiples of the a priori computable
constant eigenvectors of the flux matrix, the computations will be explicit. A similar idea has been used in~\cite{KarlsenKlingenbergRisebro:2004} to approximate a scalar conservation law with discontinuous flux. Applying an implicit-explicit discretization in time leads further to an explicit scheme contrary to e.g. \cite{MR2600931}. 
Following an approach based on hyperbolic relaxation will allow us to define a numerical scheme that does not rely on the solution to the Riemann solvers at the node of the network. It can be extended to higher-order which is also demonstrated here and existing coupling conditions can be embedded in this framework as shown in an example on traffic flow.

\section{Notation and Preliminary Discussion}

A network is a directed graph consisting of edges and vertices or junctions. We restrict the discussion to scalar hyperbolic conservation laws posed on each edge and, due to the finite speed of propagation, to the problem  at a single coupling node, that we assume at position $x=0$. The dynamics on the adjacent edge $k$ reads
\begin{equation}\label{eq:scalarconservationnetwork}
  \ddt u^k + \ddx f_k(u^k) =0 \quad \text{in }\mathcal{E}_k\times(0,\infty),\quad k \in \delta^\mp,
\end{equation}
where the state variable $u^k$ is either given on an incoming edge parameterized by $\mathcal{E}_k=(-\infty,0)$ if $k\in \delta^-=\{1,\dots,N^- \}$ or on an outgoing edge parameterized by $\mathcal{E}_k=(0, \infty)$ if $k\in\delta^+=\{N^-+1, \dots,N^-+ N^+=N\}$. Here, the flux functions are  smooth, but not-necessarily convex or concave flux functions $f_1,\dots,f_{N}:$ $\R \rightarrow \R$.  The set of all edges is denoted by $\delta^{\mp}=\delta^- \cup \delta^+$. In addition, we assume given initial data that we denote on each edge $k\in \delta^\mp$ by $u^{k,0}$.
\par 
The coupling is described by a set of conditions on the traces at the coupling node of the form
\begin{equation}\label{eq:scalarcouplingconditions}
  \Psi[u^1(0^-, t), \dots, u^{N^-}(0^-, t), u^{N^-+1}(0^+, t), \dots u^{N}(0^+, t)]=0, \quad \text{for a.e. }t>0,
\end{equation}
assuming a mapping $\Psi:\R^{N} \rightarrow \R^\ell$. The number of coupling conditions $\ell$ required to obtain a well-posed problem depends on the number of edges and the choice of flux functions, see e.g. \cite{MR3200227} and the discussion following below. 
\par 
The typical (numerical) procedure to obtain conditions on the traces of the solution $u^k$ at $x=0$ relies on a suitable Riemann solver at the junction \cite{garavelloTrafficFlowNetworks2006}. Illustrated in the case $|\delta^+|=|\delta^-|=1$, where \eqref{eq:scalarconservationnetwork} can be rewritten without edge indices as 
\begin{subequations}\label{eq:scalarconservation11}
\begin{align}
 \ddt u + \ddx f_1(u) &=0,\quad \text{in}(-\infty, 0) \times (0,\infty),\\
 \ddt u + \ddx f_2(u) &=0,\quad \text{in} (0, \infty) \times (0,\infty),
\end{align}
\end{subequations}
the idea is to connect the traces $u_0^-$ and $u_0^+$ on the left and right of the interface to the coupling data $u_L$ and $u_R$ at the interface. The coupling data are determined such that the coupling conditions $\Psi(u_L,u_R) = 0$ are satisfied and the states $u_L$ and $u_R$ can be connected to the traces $u_0^-$ on the left and $u_0^+$ on the right by means of Lax curves corresponding to characteristic fields with negative and positive characteristic speeds, respectively. This problem is formulated using parameterized Lax-curves and reduced to a typically nonlinear but finite-dimensional system of equations, see e.g.~\cite{MR3200227,HertyMuellerGerhardXiangWang:2018,GugatHertyMueller:2017}.  
\par 

To elaborate on the well-posedness and derivation of coupling conditions we introduce the Riemann solver on the edges of the network, which reads
\begin{subequations}\label{eq:scalarriemannproblem}
  \begin{equation}
    \ddt u^k + \ddx f_k(u^k) =0,\quad k \in \delta^- \cup \delta^+
  \end{equation}
      with initial data given either by  
\begin{equation}
  u^k(x,0) = \begin{cases}
               u^{k}_0 &\text{if }x\leq 0,\\
               u^k_R &\text{if }x>0,\\
             \end{cases}
             \quad \text{ if }k\in\delta^-
\end{equation}
or otherwise by
\begin{equation}
  u^k(x,0) = \begin{cases}
               u^{k}_L &\text{if }x\leq 0,\\
               u^{k}_0 &\text{if }x>0,\\
             \end{cases}
             \quad \text{ if }k\in\delta^+.
\end{equation}
\end{subequations}
On incoming edges the right initial data of the Riemann problem at time $t=0$ denoted by $u^k_R$ is unknown whereas the left data $u^k_0$ is assumed known. Analogously, on outgoing edges the left initial data of \eqref{eq:scalarriemannproblem} denoted by $u^k_L$ is unknown whereas the right data $u^{k}_0$ is given. Unknown data is obtained by a Riemann solver.

\begin{Def}[Riemann solver for scalar networks]\label{def:riemannsolver}
  A Riemann solver for problem \eqref{eq:scalarriemannproblem} is a mapping that assigns right initial data on incoming edges and left initial data on outgoing edges to given initial data on the respective opposite side, i.e., 
  \begin{align*}
    &\mathcal{RS}: \, \R^{N}\rightarrow \R^{N},\\
    &(u^{1-}_0,\dots, u^{N^-}_0,  u^{N^-+1}_0, \dots, u^{N}_0) \mapsto  (u^{1}_R,\dots, u^{N^-}_R,  u^{N^-+1}_L, \dots, u^{N}_L),
  \end{align*}
  such that (a) waves of the solution to \eqref{eq:scalarriemannproblem} have negative speed on incoming edges and positive speed on outgoing edges and (b) $\Psi[u^{1}_R,\dots, u^{N^-}_R,  u^{N^-+1}_L, \dots, u^{N}_L] = 0$.
\end{Def}
  
A necessary condition for conservation of $u$ at the coupling node is the \emph{Kirchhoff} condition: 
  \begin{equation}\label{eq:kirchhoff}
    \Psi_1[u^1_0, \dots, u^{N}_0]=   \sum_{j\in \delta^-} f_j(u^j_0) - \sum_{k\in \delta^+} f_k(u^k_0) = 0.
  \end{equation}
Hence, to conserve the quantity $u$ in the junction of the network, the coupling data obtained by the Riemann solver needs to satisfy the condition
\begin{equation}\label{eq:kirchhoffriemann}
  \sum_{j\in \delta^-} f_j(u^j_R) = \sum_{k\in \delta^+} f_k(u^k_L).
\end{equation}
In the scalar network \eqref{eq:scalarconservationnetwork} condition \eqref{eq:kirchhoff} together with admissible boundary data, see~\cite{duboisBoundaryConditionsNonlinear1988}, lead to mass conservation at the coupling node. However, these conditions are not necessarily sufficient for well-posedness of the Riemann solver and the network problem, see~\cite{MR3200227}. 

\section{Coupled Relaxation System}\label{sec:relaxation}
We follow~\cite{jinRelaxationScheme1995}, where in addition to the scalar quantity $u$ and the flux function $f$ the auxiliary variable $v \in \R$, the \emph{relaxation rate} $\eps>0$ and the \emph{relaxation speed} $\lambda>0$ have been introduced and the following system is studied:
\begin{subequations}\label{eq:relsyst}
  \begin{align}
    \ddt u^\eps + \ddx v^\eps &= 0, && \text{in }\R\times(0, \infty), \label{eq:relsystu}\\
    \ddt v^\eps + \lambda^2 \, \ddx u^\eps &= \frac 1 \eps (f(u^\eps) - v^\eps),  && \text{in }\R\times(0, \infty). \label{eq:relsystv}
  \end{align}
\end{subequations}
The relaxation system is accompanied by  initial data given through a scalar function $u^0$ and $v^0=f(u^0)$.
As $\eps\rightarrow0$ the system attains the \emph{zero relaxation limit} $(u,v) \coloneq (u^0, v^0)$ for which \eqref{eq:relsystv} necessitates the local equilibrium
$\lim_{\eps\rightarrow 0} v^\eps = v = f(u)$ and \eqref{eq:relsystu} recovers the conservation law
\begin{equation}\label{eq:scalarcl}
  \ddt u + \ddx f(u) = 0.
  \end{equation}  
The Chapman-Enskog expansion~\cite{chapmanMathematicalTheoryNonuniform1990} allows for an interpretation of \eqref{eq:relsyst} as a dissipative equation and shows that  $v^\eps = f(u^\eps) + O(\epsilon^2)$.
The \emph{subcharacteristic condition} 
\begin{equation}\label{eq:subcharacteristic}
  -\lambda \leq f^\prime(u^\eps) \leq \lambda \quad \text{for all }u^\eps.
  \end{equation}
 introduced in \cite{liuHyperbolicConservationLaws1987} guarantees that the dissipative approximation is well-posed. The dissipative equation and the conservation law were shown to have the same asymptotic behavior as the relaxation rate goes to zero in \cite{chenHyperbolicConservationLaws1994}.  An eigenvalue analysis of the relaxation system, see  Appendix~\ref{app:eigenvalues}, reveals that it is hyperbolic, has the eigenvalues $-\lambda$ and $\lambda$ and can be rewritten in terms of its characteristic variables  $w^{\eps\mp}= v^\eps/2 \mp \lambda u^\eps/2 $ as
\begin{subequations}\label{eq:characteristics}
  \begin{align}
    \ddt w^{\eps-} - \lambda \ddx w^{\eps-} &= \frac{1}{\eps} \left( f\left(  \frac{w^{\eps+}-w^{\eps-}}{\lambda}\right)  - w^{\eps+} - w^{\eps-}\right), \\
    \ddt w^{\eps+} + \lambda \ddx w^{\eps+} &= \frac{1}{\eps} \left( f\left(  \frac{w^{\eps+}-w^{\eps-}}{\lambda}\right)  - w^{\eps+} -w^{\eps-}\right).
  \end{align}
\end{subequations}

Moreover, the forward and backward Lax-curves of the relaxation system are given by straight lines in the phase plane as follows
\begin{subequations}\label{eq:laxcurves}
\begin{align}
  L^{1-}_\lambda(u^\eps_0, v^\eps_0) &= \{ (u^\eps_0 - \sigma , v^\eps_0 + \sigma \, \lambda), ~\sigma\in \R \}, \\
  L^{2+}_\lambda(u^\eps_0, v^\eps_0) &= \{ (u^\eps_0 + \sigma, v^\eps_0 + \sigma \, \lambda),~\sigma \in \R \}.
\end{align}
\end{subequations}

\subsection{Relaxation System at 1-to-1 networks}\label{sec:relsyst11}
Since the relaxation system describes the behavior of the conservation law in the relaxation limit we use it as a tool to derive suitable coupling data for the 1-to-1 coupling network. To this end we will analyze system \eqref{eq:relsyst} in the 1-to-1 coupling case in this section and derive admissible coupling data. The relaxation limit will be taken in Section \ref{sec:scheme} in case of an \emph{asymptotic preserving} discretization of the system, which was shown to converge to the correct limit as the relaxation rate tends to zero, see \cite{jin2010asymptotic} and the references therein.

We consider a 1-to-1 network, which couples two relaxation systems of the form \eqref{eq:relsyst} on a single incoming and a single outgoing edge at the coupling node, and reads
\begin{subequations}\label{eq:relsyst11}
  \begin{align}
    \ddt u^\eps + \ddx v^\eps &= 0, && \text{in }\R\setminus \{0 \}\times(0, \infty), \label{eq:relsyst11u}\\
    \ddt v^\eps + \lambda_1^2 \, \ddx u^\eps &= \frac 1 \eps (f_1(u^\eps) - v^\eps),  && \text{in }(-\infty, 0)\times(0, \infty), \label{eq:relsyst11vl}\\
    \ddt v^\eps + \lambda_2^2 \, \ddx u^\eps &= \frac 1 \eps (f_2(u^\eps) - v^\eps),  && \text{in }(0, \infty)\times(0, \infty). \label{eq:relsyst11vr}
  \end{align}
\end{subequations}
While the scalar quantity $u^\eps$ is governed by the same equation left and right from the interface the scalar auxiliary variable $v^\eps$ determining the flux of $u^\eps$ is governed by the two equations \eqref{eq:relsyst11vl} and \eqref{eq:relsyst11vr}, which account for different flux functions $f_1$ and $f_2$ left and right from the coupling node.   
Note that we choose the same relaxation rate but allow for different relaxation speeds $\lambda_1, \lambda_2 > 0$ left and right from the node. We assume that subcharacteristic conditions of the form \eqref{eq:subcharacteristic} hold at both edges. Initial data is given through a smooth and compactly supported function $u^{\eps, 0}:\R \setminus \{0\}\rightarrow \R$ and the initial condition
\begin{equation}\label{eq:relsyst11init}
u^\eps(x,0)=u^{\eps,0}(x), \quad  v^\eps(x,0)=\chi_{(-\infty, 0)}(x)f_1(u^{\eps, 0}(x)) + \chi_{(0, \infty)}(x)f_2(u^{\eps, 0}(x))
\end{equation}
for all $x\in \R\setminus \{ 0\}$ with $\chi$ denoting the characteristic function. To close the system coupling conditions of the form 
\begin{equation}\label{eq:relsyst11coupling}
  \Psi\left[(u^{\eps}(0^-, t), v^{\eps}(0^-, t)), (u^{\eps}(0^+,t),v^{\eps}(0^+, t))\right]=0 \quad \text{for a.e. }t>0
\end{equation}
taking into account both variables of the system are required. Since a linear system of two conservation laws is given, two conditions are imposed for well-posedness, i.e. $\Psi:\R^4\rightarrow \R^2$.  At a fixed time $t>0$ we denote the traces left and right from the coupling node by $u_0^{\eps-}$, $v_0^{\eps-}$, $u_0^{\eps+}$ and $u_0^{\eps-}$. Analogously to Definition \ref{def:riemannsolver} we formally define a Riemann solver for the component-wise Riemann problem on the 1-to-1 network assigning coupling data to the traces, i.e.,
\begin{equation}\label{eq:riemannsolverrelsyst11}
\mathcal{RS}_\text{rel}: (u_0^{\eps-}, v_0^{\eps-}, u_0^{\eps+}, v_0^{\eps-}) \mapsto (u_R^{\eps}, v_R^{\eps}, u_L^{\eps}, v_L^{\eps}).
\end{equation}
In the following, the construction of \eqref{eq:riemannsolverrelsyst11} is discussed in detail. To obtain admissible boundary data, $(u_R^{\eps}, v_R^{\eps})$ needs to be connected to $(u_0^{\eps-}, v_0^{\eps-})$ by a wave with negative velocity, whereas  $(u_L^{\eps}, u_L^{\eps})$ needs to be connected to  $(u_0^{\eps+}, u_0^{\eps-})$ by a wave with positive velocity as shown in Figure~\ref{fig:wavestructure}. Thus, by the eigenvalue analysis of the relaxation system, we require the two conditions
\begin{equation}\label{eq:relaxationwaves}
  (u_R^{\eps}, v_R^{\eps}) \in L^{1-}_{\lambda_1}(u_0^{\eps-}, v_0^{\eps-}) \quad \text{and}\quad  (u_L^{\eps}, v_L^{\eps}) \in L^{2+}_{\lambda_2}(u_0^{\eps+}, v_0^{\eps+}).
\end{equation}
Well-posedness of \eqref{eq:riemannsolverrelsyst11} is obtained taking into account the two coupling conditions, which must be satisfied by the coupling data obtained by the Riemann solver. As the first coupling condition we impose the Kirchhoff condition \eqref{eq:kirchhoff} in \eqref{eq:relsyst11u} and obtain 
\begin{subequations}\label{eq:relsyst11couplingconditions}
  \begin{equation}
\Psi_1\left[ (u_R^{\eps}, v_R^{\eps}), (u_L^{\eps}, v_L^{\eps}) \right] = v_R^{\eps} - v_L^{\eps} = 0.
\end{equation}
Similarly, to conserve the mass of the auxiliary variable at the coupling node in the relaxation limit, we impose as second coupling condition \eqref{eq:kirchhoff} in  \eqref{eq:relsyst11vl} and \eqref{eq:relsyst11vr} and get the condition
  \begin{equation}
\Psi_2\left[(u_R^{\eps}, v_R^{\eps}), (u_L^{\eps}, v_L^{\eps})\right] = \lambda_1^2 \, u_R^{\eps} - \lambda_2^2 u_L^{\eps} = 0. 
\end{equation}
\end{subequations}
Combining \eqref{eq:relaxationwaves} and \eqref{eq:relsyst11couplingconditions} a regular linear system is obtained, which determines \eqref{eq:riemannsolverrelsyst11}. In explicit form the coupling data is given by 
\begin{subequations}\label{eq:relaxationcouplingdata}
  \begin{align}
    u^\eps_R &= \frac{\lambda_2}{\lambda_1} \, \frac{ \lambda_1 u^{\eps-}_0 + \lambda_2 u^{\eps+}_0  +  v^{\eps-}_0 - v^{\eps+}_0}{\lambda_1 + \lambda_2},\quad
    u^\eps_L = \frac{\lambda_1}{\lambda_2} \frac{\lambda_1 u^{\eps-}_0 + \lambda_2 u^{\eps+}_0 +  v^{\eps-}_0 -  v^{\eps+}_0}{\lambda_1 + \lambda_2}, \label{eq:relaxationcouplingdatau}\\
    v^\eps_R &= v^\eps_L = \frac{\lambda_1 v^{\eps-}_0 + \lambda_2 v^{\eps+}_0 + \lambda_1^2 u^{\eps-}_0 - \lambda_2^2 u^{\eps+}_0}{\lambda_1 + \lambda_2}. \label{eq:relaxationcouplingdatav}
  \end{align}
\end{subequations}

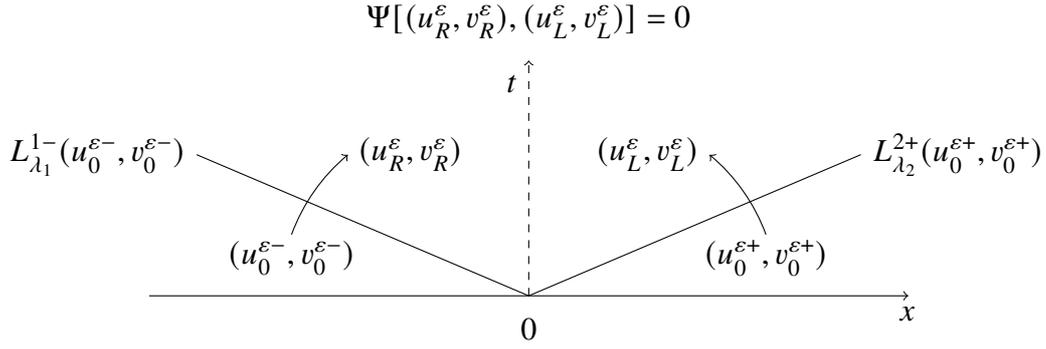
\begin{figure}
  \centering
  \begin{tikzpicture}[x=\linewidth/25,y=\linewidth/25]
\draw[->, thin] (-8,0) -- (8,0) node[below] {$x$};
\node[label=below:{$0$}] at (0, 0){};
\draw[->, thin, dashed] (0, 0) -- (0, 5) node[label=above:{$\Psi[(u_R^\eps, v_R^\eps),  (u_L^\eps, v_L^\eps)]=0$}] {};
\node[left] at (0, 4.5){$t$};
\node at (-5, .8){$(u_0^{\eps-}, v_0^{\eps-})$};
\draw[->] (-5, 1.3) arc (160:130:4);
\node at (-2.5, 3){$(u_R^{\eps}, v_R^{\eps})$};
\node at (5, .8){$(u_0^{\eps+}, v_0^{\eps+})$};
\draw[->] (5, 1.3) arc (20:50:4);
\node at (2.5, 3){$(u_L^{\eps}, v_L^{\eps})$};
\draw[-] (0, 0) -- (-7, 3) node[left] {$L^{1-}_{\lambda_1}(u_0^{\eps-}, v_0^{\eps-})$};
\draw[-] (0, 0) -- (7, 3) node[right] {$L^{2+}_{\lambda_2}(u_0^{\eps+}, v_0^{\eps+})$};
\end{tikzpicture}
  \caption{Wave structure of the coupled relaxation system on the 1-to-1 network in the $x$-$t$-plane. Incoming traces $(u_0^{\eps-}, v_0^{\eps-})$ are connected to the coupling data $(u_R^{\eps}, v_R^{\eps})$ by the $L^{1-}_{\lambda_1}$ backward Lax curve and outgoing traces $(u_0^{\eps+}, v_0^{\eps+})$ are connected to the coupling data $(u_L^{\eps}, v_L^{\eps})$ by the $L^{2+}_{\lambda_2}$ forward Lax curve. Coupling data on the incoming and outgoing edges are related by the coupling condition $\Psi$.}\label{fig:wavestructure}
\end{figure}

\begin{Rmk}\label{rem:equalspeeds}
  Defining $\lambda = \max \{ \lambda_1, \lambda_2\}$ system \eqref{eq:relsyst11} can be rewritten using relaxation speed $\lambda$ in both \eqref{eq:relsyst11vl} and \eqref{eq:relsyst11vr}. In this case we obtain the simplified coupling data
\begin{subequations}\label{eq:boundaryvalsimple}
  \begin{align}
    u^\eps_R &= u^\eps_L = \frac{u^{\eps-}_0 + u^{\eps+}_0}{2} + \frac{v^{\eps-}_0 - v^{\eps+}_0}{2 \lambda}, \\
    v^\eps_R &= v^\eps_L = \frac{v^{\eps-}_0 + v^{\eps+}_0}{2} + \frac{\lambda}{2}\left( u^{\eps-}_0 - u^{\eps+}_0\right).
  \end{align}
\end{subequations}
A drawback of this adjustment is the increase of numerical diffusion in the schemes discussed in Section~\ref{sec:scheme}.
\end{Rmk}

%
\subsection{Central Scheme for 1-to-1 Networks}\label{sec:scheme}
In this section we derive the central scheme for 1-to-1 networks of scalar conservation laws. Therefore we will start in Sections \ref{sec:semidiscrete} and \ref{sec:relscheme11} from a semi-discretization of the coupled relaxation system making use of the derived coupling data from Section \ref{sec:relsyst11}. A time discretization is introduced in Section \ref{sec:imex} for which the relaxation limit is considered in Section \ref{sec:limit}. In Section \ref{sec:highorder} we introduce a second order scheme.

\subsubsection{Semi-discrete Scheme}\label{sec:semidiscrete}
We introduce a uniform grid on the real line by fixing $\Delta x > 0$ and defining the mesh points $x_{j-1/2}= j \,\Delta x $ for any $j \in \Z$. We denote the approximate average of any scalar quantity $q$ in the cell $I_j\coloneq[x_{j}, x_{j+1}]$, which still depends on the time variable, by $q_j$.

We  obtain a scheme for the relaxation system in characteristic variables \eqref{eq:characteristics} by applying the first order upwind discretization, see e.g., \cite{levequeFiniteVolumeMethods2002}. Hereby we obtain by the signs of the eigenvalues for any $j\in \Z$
\begin{subequations}\label{eq:upwindcharvar}
  \begin{align}
            \partial_t w^{\eps-}_j - \frac{\lambda}{\Delta x} (w^{\eps-}_{j+1} - w^{\eps-}_{j}) &= \frac{1}{\eps}\left( f\left(  \frac{w^{\eps+}_j-w^{\eps-}_j}{\lambda}\right)  - w^{\eps+}_j-w^{\eps-}_j\right),\\
    \partial_t w^{\eps+}_j + \frac{\lambda}{\Delta x} (w^{\eps+}_j - w^{\eps+}_{j-1}) &= \frac{1}{\eps}\left( f\left(  \frac{w^{\eps+}_j-w^{\eps-}_j}{\lambda}\right)  - w^{\eps+}_j-w^{\eps-}_j\right).
  \end{align}
\end{subequations}
To derive \eqref{eq:upwindcharvar}, we have additionally applied a midpoint discretization to the flux function to approximate $f(q_j) \approx \int_{I_j}f(q(t, x)) \, dx$. Transforming back to the original variables $u_\eps = (w^{\eps+} - w^{\eps-})/\lambda$ and $v_\eps = w^{\eps+} + w^{\eps-}$ we end up with a semi-discrete scheme for \eqref{eq:relsyst}, that reads for any $j\in\Z$
\begin{subequations}\label{eq:relaxationupwindscheme}
  \begin{align}
    \partial_t u^\eps_j + \frac{v^\eps_{j+1} - v^\eps_{j-1}}{2 \Delta x} - \frac{\lambda}{2 \Delta x} \left(  u^\eps_{j+1} -2 u^\eps_j + u^\eps_{j-1} \right) &= 0, \\
    \partial_t v^\eps_j + \frac{\lambda^2}{2 \Delta x} \left(  u^\eps_{j+1} - u^\eps_{j-1} \right) - \frac{\lambda}{2 \Delta x}  \left(  v^\eps_{j+1} -2 v^\eps_j + v^\eps_{j-1} \right) &= \frac{1}{\eps} \left( f(u^\eps_j) - v^\eps_j \right).
  \end{align}
\end{subequations}

\subsubsection{Coupled scheme}\label{sec:relscheme11}
\begin{figure}
  \begin{tikzpicture}[x=\linewidth/25,y=\linewidth/25]
\draw[->, thin] (-12,0) -- (12,0) node[right] {$x$};
\node[label=below:$x_{-5/2}$] at (-10, 0){$|$};
\node[label=below:$I_{-2}$] at (-7.5, 0){};
\node[label=below:$x_{-3/2}$] at (-5, 0){$|$};
\node[label=below:$I_{-1}$] at (-2.5, 0){};
\node[label=below:{$x_{-1/2}=0$}] at (0, 0){$|$};
\node[label=below:$I_0$] at (2.5, 0){};
\node[label=below:$x_{1/2}$] at (5, 0){$|$};
\node[label=below:$I_1$] at (7.5, 0){};
\node[label=below:$x_{3/2}$] at (10, 0){$|$};
\draw[-, thin, dashed] (0, 0) -- (0, 4) node[label=above:{$\Psi[(u_0^{\eps-}, v_0^{\eps-}), (u_0^{\eps+}, v_0^{\eps+})]=0$}] {};
\coordinate (A) at (-5,1);
\node[label=above:{$\begin{aligned}
                     \ddt u^\eps + \ddx v^\eps &= 0\\
                     \ddt v^\eps + \lambda_1^2 \, \ddx u^\eps &= \frac 1 \eps (f_1(u^\eps) - v^\eps)
                     \end{aligned}$}] at (-6, 1) {};
\node[label=above:{$\begin{aligned}
                     \ddt u^\eps + \ddx v^\eps &= 0\\
                     \ddt v^\eps + \lambda_2^2 \, \ddx u^\eps &= \frac 1 \eps (f_2(u^\eps) - v^\eps)
                     \end{aligned}$}] at (6, 1) {};
\end{tikzpicture}
  \caption{The relaxation system in the 1-to-1 coupling case on the discretized real line.}\label{fig:one-to-one-num}
\end{figure}

We consider a discretization of the coupled relaxation system \eqref{eq:relsyst11}. Figure \ref{fig:one-to-one-num} shows an illustration of the setting and the space discretization. We denote by $\dots, u_{-2}^\eps, u_{-1}^\eps$ and $\dots, v_{-2}^\eps, v_{-1}^\eps$ discretizations of $u^{\eps}$ and $v^{\eps}$  left from the coupling node and by $u_0^\eps, u_1^\eps,\dots$ and $v_0^\eps, v_1^\eps,\dots$ discretizations of $u^{\eps}$ and $v^{\eps}$ right from the coupling node, respectively. Now we apply scheme~\eqref{eq:relaxationupwindscheme} and use the coupling data derived in Section \ref{sec:relsyst11} as ghost cell averages beyond the coupling node when approaching $x=0$ from the left and from the right. Thus we obtain
\begin{subequations}\label{eq:relaxationupwindschemeleft}
  \begin{align}
        \ddt u^\eps_{-1} + \frac{v^\eps_R- v^\eps_{-2}}{2 \Delta x} - \frac{\lambda_1}{2 \Delta x} \left(  u^\eps_R -2 u^\eps_{-1} + u^\eps_{-2} \right) &= 0, \\
    \ddt v^\eps_{-1} + \frac{\lambda_1^2}{2 \Delta x} \left(  u^\eps_R - u^\eps_{-2} \right) - \frac{\lambda_1}{2 \Delta x}  \left(  v^\eps_R -2 v^\eps_{-1} + v^\eps_{-2} \right) &= \frac{1}{\eps} \left( f_1(u^\eps_{-1}) - v^\eps_{-1} \right)
  \end{align}
\end{subequations}
for the evolution of the cell average left from the coupling node and
\begin{subequations}\label{eq:relaxationupwindschemeright}
  \begin{align}
    \ddt u^\eps_0 + \frac{v^\eps_{1} - v^\eps_L}{2 \Delta x} - \frac{\lambda_2}{2 \Delta x} \left(  u^\eps_{1} -2 u^\eps_0 + u^\eps_{L} \right) &= 0, \\
    \ddt v^\eps_0 + \frac{\lambda_2^2}{2 \Delta x} \left(  u^\eps_{1} - u^\eps_{L} \right) - \frac{\lambda_2}{2 \Delta x}  \left(  v^\eps_{1} -2 v^\eps_0 + v^\eps_{L} \right) &= \frac{1}{\eps} \left( f_2(u^\eps_{0}) - v^\eps_0 \right)                                                                                                                                                                         
    \end{align}
  \end{subequations}
 for the evolution of the cell average right from the coupling node. Clearly, \eqref{eq:relaxationupwindschemeleft} and \eqref{eq:relaxationupwindschemeright} can be complemented to a scheme over the full real line by additionally considering \eqref{eq:relaxationupwindscheme} for $j\in\Z \setminus \{-1, 0 \}$  with $\lambda$ and $f$ substituted by $\lambda_1$ and $f_1$ for negative $j$ and $\lambda_2$ and $f_2$ for positive $j$.
  
 In the discretized setting traces are obtained from the cell averages next to the coupling node. Thus we have $u_0^{\eps-}=u_{-1}^\eps$, $u_0^{\eps+}=u_{0}^\eps$, $v_0^{\eps-}=v_{-1}^\eps$ and $v_0^{\eps+}=v_{0}^\eps$. Substituting now the coupling data~\eqref{eq:relaxationcouplingdata} into \eqref{eq:relaxationupwindschemeleft} and \eqref{eq:relaxationupwindschemeright} we obtain in case of $\lambda=\lambda_1=\lambda_2$ (clf. Remark \ref{rem:equalspeeds}) left from the coupling node
\begin{subequations}\label{eq:relaxationupwindschemeleftonelambda}
  \begin{align}
    \ddt u^\eps_{-1} + \frac{v^\eps_0 - v^\eps_{-2}}{2 \Delta x} - \frac{\lambda}{2 \Delta x} \left(  u^\eps_0 -2 u^\eps_{-1} + u^\eps_{-2} \right) &= 0, \\
    \ddt v^\eps_{-1} + \frac{\lambda^2}{\Delta x} \left( u^\eps_0 - u^\eps_{-2} \right) - \frac{\lambda}{2 \Delta x}  \left( v^\eps_0 -2 v^\eps_{-1} + v^\eps_{-2} \right) &= \frac{1}{\eps} \left( f_1(u^\eps_{ -1}) - v^\eps_{-1} \right)
  \end{align}
\end{subequations}
and right from the coupling node
\begin{subequations}\label{eq:relaxationupwindschemerightonelambda}
  \begin{align}
    \ddt u^\eps_0 + \frac{v^\eps_{1} - v^\eps_{-1}}{2 \Delta x} - \frac{\lambda}{2 \Delta x} \left(  u^\eps_{1} -2 u^\eps_0 + u^\eps_{-1} \right) &= 0, \\
    \ddt v^\eps_0 + \frac{\lambda^2}{2 \Delta x} \left(  u^\eps_{1} - u^\eps_{-1} \right) - \frac{\lambda}{2 \Delta x}  \left(  v^\eps_{1} -2 v^\eps_0 + v^\eps_{-1} \right) &= \frac{1}{\eps}  \left( f_2(u^\eps_{0}) - v^\eps_{0} \right).                                                                                                                    
    \end{align}
  \end{subequations}
  The corresponding evolution formulas in case of different relaxation speeds $\lambda_1\neq \lambda_2$ are given in Appendix~\ref{app:speeds}. The following consistency result follows from \eqref{eq:relaxationupwindschemeleftonelambda} and \eqref{eq:relaxationupwindschemerightonelambda}.
  
    \begin{Prop}[Consistency of the semi-discrete scheme]\label{prop:consistency}
    We assume $f=f_1=f_2$ as well as $\lambda_1=\lambda_2=\lambda$. Then the semi-discrete scheme for the 1-to-1 network system \eqref{eq:relsyst11}, which consists of \eqref{eq:relaxationupwindscheme} for $j \in \Z \setminus \{-1, 0\}$ supplemented by \eqref{eq:relaxationupwindschemeleft}, \eqref{eq:relaxationupwindschemeright} and coupling data \eqref{eq:relaxationcouplingdata}, is identical to the semi-discrete scheme for the uncoupled relaxation system \eqref{eq:relsyst} given by \eqref{eq:relaxationupwindscheme} for $j \in \Z$.
  \end{Prop}
  
  \subsubsection{Fully discrete scheme}\label{sec:imex}
In this section we derive an implicit--explicit scheme and consider to this end a uniform partition of the time line by introducing the time increment $\Delta t>0$ and setting $t^n= n \Delta t$ for all $n \in \N_0$. The approximate average in the cell $I_j$ of any scalar quantity $q$ at time $t^n$ is denoted by $q_j^n$. For $j \in \Z \setminus \{ -1, 0\}$ an implicit-explicit time discretization of \eqref{eq:relaxationupwindscheme} is given by
\begin{subequations}\label{eq:unsplit}
  \begin{align}
    u^{\eps, n+1}_j &= u^{\eps, n}_j - \frac{\Delta t}{2 \Delta x}\left( v^{\eps, n}_{j+1} - v^{\eps, n}_{j-1}\right) + \frac{\lambda_{\alpha(j)} \Delta t}{2 \Delta x} \left( u^{\eps, n}_{j+1} - 2 u^{\eps, n}_{j} + u^{\eps, n}_{j-1}\right) , \label{eq:unsplitu}\\
    v^{\eps, n+1}_j &= v^{\eps, n}_j - \frac{\lambda_{\alpha(j)}^2 \Delta t}{2 \Delta x}\left( u^{\eps, n}_{j+1} - u^{\eps, n}_{j-1}\right) + \frac{\lambda_{\alpha(j)} \Delta t}{2 \Delta x}\left( v^{\eps, n}_{j+1} - 2 v^{\eps, n}_j + v^{\eps, n}_{j-1}\right) \notag \\
    &\quad + \frac 1 \eps \left( f_{\alpha(j)}(u^{\eps, n+1}_j) - v^{\eps, n+1}_j \right),\label{eq:unsplitv}
    \end{align}
  \end{subequations}
   where $\alpha(j) = 1$ for negative $j$ and $\alpha(j) = 2$ for positive $j$.
  This scheme was proposed in \cite{huAsymptoticPreservingSchemesMultiscale2017} and while it handles most terms explicitly the stiff relaxation term is treated implicitly for increased stability. Since at each solution update $u_j^{\eps,n+1}$ can be computed first by the explicit evolution formula \eqref{eq:unsplitu}, it is not necessary to solve a nonlinear system to evaluate \eqref{eq:unsplitv} afterwards.

  We complement this scheme left from the coupling node by
  \begin{subequations}\label{eq:unsplitcoupling}
    \begin{align}
        u^{\eps, n+1}_{-1} &= u^{\eps, n}_{-1} - \frac{\Delta t}{2 \Delta x}\left( v^{\eps, n}_{R} - v^{\eps, n}_{-2}\right) + \frac{\lambda_1 \Delta t}{2 \Delta x} \left( u^{\eps, n}_{R} - 2 u^{\eps, n}_{-1} + u^{\eps, n}_{-2}\right), \\
      v^{\eps, n+1}_{-1} &= v^{\eps, n}_{-1} - \frac{\lambda_1^2 \Delta t}{2 \Delta x}\left( u^{\eps, n}_{R} - u^{\eps, n}_{-2}\right) + \frac{\lambda_1 \Delta t}{2 \Delta x}\left( v^{\eps, n}_{R} - 2 v^{\eps, n}_{-1} + v^{\eps, n}_{-2}\right) + \frac 1 \eps \left( f_1(u^{\eps, n+1}_{-1}) - v^{\eps, n+1}_{-1} \right),\label{eq:unsplitvl}
    \end{align}
    and right from the coupling node by
    \begin{align}
            u^{\eps, n+1}_{0} &= u^{\eps, n}_0 - \frac{\Delta t}{2 \Delta x}\left( v^{\eps, n}_{1} - v^{\eps, n}_{L}\right) + \frac{\lambda_2 \Delta t}{2 \Delta x} \left( u^{\eps, n}_{1} - 2 u^{\eps, n}_{0} + u^{\eps, n}_{L}\right), \\
      v^{\eps, n+1}_{0} &= v^{\eps, n}_0 - \frac{\lambda_2^2 \Delta t}{2 \Delta x}\left( u^{\eps, n}_{1} - u^{\eps, n}_{L}\right) + \frac{\lambda_2 \Delta t}{2 \Delta x}\left( v^{\eps, n}_{1} - 2 v^{\eps, n}_0 + v^{\eps, n}_{L}\right) + \frac 1 \eps \left( f_2(u^{\eps, n+1}_0) - v^{\eps, n+1}_0 \right). \label{eq:unsplitvr}
    \end{align}
  \end{subequations}
  As \eqref{eq:unsplitcoupling} is a time discretization of \eqref{eq:relaxationupwindschemeleft} and \eqref{eq:relaxationupwindschemeright} the coupling data accounting for boundary information in \eqref{eq:unsplitcoupling} is derived from cell averages next to the coupling node by the Riemann solver \eqref{eq:riemannsolverrelsyst11} as
  \begin{equation}
    \mathcal{RS}_\text{rel}(u_{-1}^{\eps,n}, v_{-1}^{\eps,n}, u_0^{\eps,n}, v_0^{\eps,n}) \eqcolon (u_R^{\eps,n}, v_R^{\eps,n}, u_L^{\eps,n}, u_L^{\eps,n}).
  \end{equation}
  \begin{Rmk} In \cite{huAsymptoticPreservingSchemesMultiscale2017} the authors further consider an alternative two-stage time discretization, which treats the relaxation implicitly in a similar fashion as \eqref{eq:unsplit}. The limit scheme derived in Section \ref{sec:limit} is also obtained when this alternative time discretization is used.
  \end{Rmk}
  
  \subsubsection{Relaxation Limit}\label{sec:limit}
  In this section we derive the relaxation limit of the scheme for the 1-to-1 relaxation system \eqref{eq:relsyst11} introduced in Section \ref{sec:imex} given by \eqref{eq:unsplit} and \eqref{eq:unsplitcoupling}. Our aim is to obtain in this way a scheme for scalar conservation laws in the 1-to-1 network case \eqref{eq:scalarconservation11}.

  For the limit process an asymptotic expansion at the relaxation state of the state variables can be considered. As in the scheme the relaxation time appears only in the discretized balance term, the following procedure is equivalent: we keep $\Delta x$ and $\Delta t $ fixed and assume that the magnitude of these quantities as well as of the occurring cell averages are independent of the relaxation time and then consider the limit $\eps\rightarrow 0$ \cite{huAsymptoticPreservingSchemesMultiscale2017}. From \eqref{eq:unsplitv}, \eqref{eq:unsplitvl} and \eqref{eq:unsplitvr} we get
  \begin{equation}\label{eq:vjlimit}
    v_j^{n+1} = f_1(u_j^{n+1}) \quad\text{ if }j \leq -1 \quad\text{and}\quad  v_j^{n+1} = f_2(u_j^{n+1}) \quad \text{ if }0 \leq j 
  \end{equation}
  for any $n\in \N_0$, where we have used the limit notations $u_j^{n+1} = \lim_{\eps\rightarrow 0} u_j^{\eps, n+1}$ and $v_j^n = \lim_{\eps\rightarrow 0} v_j^{\eps, n+1}$. Consequently, the evolution formulas for the auxiliary variable can be discarded and we obtain the limit scheme
  \begin{subequations}\label{eq:limitscheme}
    \begin{align}
      u^{n+1}_j &= u^{n}_j - \frac{\Delta t}{2 \Delta x}\left( f_{\alpha(j)}(u^{n}_{j+1}) - f_{\alpha(j)}(u^{n}_{j-1})\right) + \frac{\lambda_{\alpha(j)} \Delta t}{2 \Delta x} \left( u^{n}_{j+1} - 2 u^{n}_{j} + u^{n}_{j-1}\right),\label{eq:limitschemeg}\\
      u^{n+1}_{-1} &= u^{n}_{-1} - \frac{\Delta t}{2 \Delta x}\left( v^{n}_{R} - f_1(u^{n}_{-2})\right) + \frac{\lambda_1 \Delta t}{2 \Delta x} \left( u^{n}_{R} - 2 u^{n}_{-1} + u^{n}_{-2}\right), \label{eq:limitschemel}\\
         u^{n+1}_{0} &= u^{n}_{0} - \frac{\Delta t}{2 \Delta x}\left( f_2(u^n_{1}) - v^{n}_{L}\right) + \frac{\lambda_2 \Delta t}{2 \Delta x} \left( u^{n}_{1} - 2 u^{n}_{0} + u^{n}_{L}\right) \label{eq:limitschemer}
      \end{align}
    \end{subequations}
    for $j\in \Z \setminus\{-1, 0\}$ in \eqref{eq:limitschemeg} and with $\alpha(j) = 1$ for negative $j$ and $\alpha(j) = 2$ for positive $j$. We emphasize that in general $v^n_R \neq f_1(u^n_R)$ and $v^n_L\neq f_2(u^n_L)$ in \eqref{eq:limitschemel} and \eqref{eq:limitschemer} and that, although the scheme approximates the single scalar quantity $u$, coupling data of the auxiliary variable is required for its evaluation. We further note that the limit scheme is fully explicit. Coupling data in \eqref{eq:limitscheme} is determined by the Riemann solver \eqref{eq:riemannsolverrelsyst11} as
  \begin{equation}\label{eq:limitschemers}
    \mathcal{RS}_\text{rel}(u_{-1}^{n}, f_1(u_{-1}^{n}), u_0^{n}, f_2(u_0^{n})) \eqcolon (u_R^{n}, v_R^{n}, u_L^{n}, v_L^{n}).
  \end{equation}
  
If we assume $\lambda=\lambda_1=\lambda_2$ and substitute the coupling data \eqref{eq:relaxationcouplingdata} derived in Section \ref{sec:relsyst11} taking into account \eqref{eq:limitschemers}, we obtain
\begin{subequations}\label{eq:limitschemesub}
  \begin{align}
      u^{n+1}_{-1} &= u^{n}_{-1} - \frac{\Delta t}{2 \Delta x}\left( f_2(u^n_0) - f_1(u^{n}_{-2})\right) + \frac{\lambda_1 \Delta t}{2 \Delta x} \left( u^{n}_0 - 2 u^{n}_{-1} + u^{n}_{-2}\right), \label{eq:limitschemesubl}\\
         u^{n+1}_{0} &= u^{n}_{0} - \frac{\Delta t}{2 \Delta x}\left( f_2(u^n_{1}) - f_1(u^{n}_{-1})\right) + \frac{\lambda_2 \Delta t}{2 \Delta x} \left( u^{n}_{1} - 2 u^{n}_{0} + u^{n}_{-1}\right), \label{eq:limitschemesubr}                                                                                                              
    \end{align}
  \end{subequations}
  which replace \eqref{eq:limitschemel} and \eqref{eq:limitschemer} in the scheme \eqref{eq:limitscheme}. This makes evident that the limit scheme can be written in the conservative form
  \begin{subequations}\label{eq:limitschemeconservative}
\begin{equation}\label{eq:conservativeform}
  u_j^{n+1} = u_j^n - \frac{\Delta t }{\Delta x} \left( F_{j+1/2}^n -  F_{j-1/2}^n\right) \quad \text{for all }j \in \Z
\end{equation}
using numerical fluxes that depend on the two cell averages next to the cell interface $x_j$ and read
\begin{equation}
   F_{j-1/2}^n = \begin{cases}
                 \frac 1 2 \, (f_1(u_{j-1}^n) + f_1(u_{j}^n))  - \frac \lambda 2 (u_j^n-u_{j-1}^n)  &\text{if }j<0, \\[5pt]
                 \frac 1 2 \, ( f_1(u_{-1}^n) + f_2(u_{0}^n))  - \frac \lambda 2 (u_0^n-u_{-1}^n) &\text{if }j=0, \\[5pt]
                 \frac 1 2 \, ( f_2(u_{j-1}^n) + f_2(u_{j}^n))  - \frac \lambda 2 (u_j^n-u_{j-1}^n) &\text{if }j>0. 
               \end{cases}
\end{equation}
\end{subequations}
We provide evolution formulas corresponding to \eqref{eq:limitschemesub} and \eqref{eq:limitschemeconservative} for differing relaxation speeds in Appendix \ref{app:speeds}. In case $f_1=f_2$ scheme \eqref{eq:limitschemeconservative} is the first order relaxed scheme from \cite{jinRelaxationScheme1995}, which is identical to the Lax--Friedrich scheme if we additionally assume $\lambda=\frac{\Delta x}{\Delta t}$. Thus the consistency property from Proposition \ref{prop:consistency} transfers to the limit scheme as follows. 
\begin{Prop}[Consistency of the limit scheme]\label{prop:consistencylimit}
  Let $f_1=f_2=f$ and $\lambda_1=\lambda_2=\lambda$. Then the limit scheme \eqref{eq:limitscheme} using coupling data \eqref{eq:limitschemers}, \eqref{eq:relaxationcouplingdata} is identical to the first order relaxed scheme (see \cite{jinRelaxationScheme1995}) for the scalar conservation law.
\end{Prop}

The relaxed scheme for the uncoupled conservation law was analyzed regarding monotonicity and the result carries over directly to the coupled scheme in the relaxation limit.  

\begin{Prop}[\cite{jinRelaxationScheme1995, crandallMethodFractionalSteps1980}]
 Let $f_1=f_2=f$, $\lambda_1=\lambda_2=\lambda$, $\lambda \Delta t \leq \Delta x$ and assume that the subcharacteristic condition \eqref{eq:subcharacteristic} holds. Then the limit scheme \eqref{eq:limitscheme} using coupling data \eqref{eq:limitschemers}, \eqref{eq:relaxationcouplingdata} is a consistent and monotone scheme (see e.g., \cite{godlewskiNumericalApproximationHyperbolic1996}) and therefore converges to the entropy solution of \eqref{eq:scalarcl} as both $\Delta t$ and $\Delta x$ tend to zero.
\end{Prop}

Due to \eqref{eq:limitschemeconservative} and its generalization in Appendix \ref{app:speeds} we state the following result about the conservation of the total mass $\sum_{j\in\Z} u_j^n$ over time. Conservation at the coupling node for general networks is defined in Section \ref{sec:networkschemes}, whereas it follows directly from the conservative form in the 1-to-1 case.

\begin{Cor}
   The limit scheme for the 1-to-1 network \eqref{eq:scalarconservation11} given by \eqref{eq:limitscheme} and coupling data \eqref{eq:limitschemers}, \eqref{eq:relaxationcouplingdata} is conservative at the coupling node. It is further conservative over the full real line in the case $f_1(0)=f_2(0)$ and over bounded domains if zero-flux boundary conditions are imposed.
\end{Cor}

Other schemes for conservation laws in the 1-to-1 network case have been introduced in the literature. In particular, the work \cite{MR2045460} analyzes a similar scheme, which can be written as \eqref{eq:limitscheme} with coupling data $u_R^n=u_0^n$, $v_R^n=f_1(u_0^n)$, $u_L^n=u_{-1}^n$ and $v_R^n=f_2(u_{-1}^n)$. Unlike the method we introduce here, this scheme does not admit a conservative form but it has been proven to converge.

\subsubsection{Second-order Extension}\label{sec:highorder}
In this section we are concerned with a second order scheme for the 1-to-1 scalar network. The scheme is derived by extending the approach in Sections \ref{sec:semidiscrete}--\ref{sec:limit} using piece-wise linear approximations left and right from the coupling node. A similar extension is considered in \cite{jinRelaxationScheme1995} for uncoupled conservation laws.

A second order scheme for system \eqref{eq:relsyst} is derived by applying the MUSCL scheme \cite{vanleerUltimateConservativeDifference1979} to the upwind discretization of the relaxation system in characteristic variables \eqref{eq:upwindcharvar}. This way we obtain the semi-discrete scheme 
\begin{subequations}\label{eq:charvalMUSCL}
  \begin{align}
    \partial_t w^{\eps-}_j - \frac{\lambda}{\Delta x} (w^{\eps-}_{j+1/2} - w^{\eps-}_{j-1/2}) &= \frac{1}{\eps}\left( f\left(  \frac{w^{\eps+}_j-w^{\eps-}_j}{\lambda}\right)  - w^{\eps+}_j-w^{\eps-}_j\right),\\
    \partial_t w^{\eps+}_j + \frac{\lambda}{\Delta x} ( w^{\eps+}_{j+1/2} -  w^{\eps+}_{j-1/2}) &= \frac{1}{\eps}\left( f\left(  \frac{w^{\eps+}_j-w^{\eps-}_j}{\lambda}\right)  - w^{\eps+}_j-w^{\eps-}_j\right)
  \end{align}
\end{subequations}
for all $j\in \Z$, which makes use of interface reconstructions. These are obtained by extrapolation from the upwind direction, i.e.,
\begin{equation}
   w^{\eps-}_{j-1/2} = w^{\eps-}_{j} - \frac{\Delta x}{2} s_{j}^-, \quad  w^{\eps+}_{j-1/2} = w^{\eps+}_{j-1} + \frac{\Delta x}{2} s_{j-1}^+ \quad \text{for all }j\in \Z. 
\end{equation}
The slope in each cell is given by the \emph{monotonized central-difference limiter}
\begin{equation}\label{eq:mclimiter}
  s_j^\mp = \operatorname{minmod}\left( 2\frac{w^{\eps\mp}_{j}  - w^{\eps\mp}_{j-1} }{\Delta x}, ~\frac{w^{\eps\mp}_{j+1}  - w^{\eps\mp}_{j-1} }{2\,\Delta x}, ~ 2 \frac{w^{\eps\mp}_{j+1}  - w^{\eps\mp}_{j} }{\Delta x} \right).
\end{equation}
The selected limiter \eqref{eq:mclimiter} was introduced in \cite{vanleerUltimateConservativeDifference1979} and designed to yield sharp resolutions near discontinuities. In smooth regions it admits the central difference whereas the accuracy at non-sonic critical points is reduced to preserve the monotonicity of the discrete solution.
The \emph{minmod} operator used in its formulation is defined by
\begin{equation}\label{eq:minmod}
\operatorname{minmod}(q_1,\dots,q_n)=
\begin{cases}
\max\{q_1,\dots,q_n\}	&\text{if } q_k<0,~k=1,\dots,n,\\
\min\{q_1,\dots,q_n\}	&\text{if } q_k>0,~k=1,\dots,n, \\
0						&\text{otherwise}.
\end{cases}
\end{equation}

As in the derivation of scheme \eqref{eq:relaxationupwindscheme} we transform \eqref{eq:charvalMUSCL} back to the original variables of the relaxation system, and get
\begin{subequations}\label{eq:relaxationMUSCLupwindscheme}
  \begin{align}
    \partial_t u^\eps_j + \frac{v^\eps_{j+1} - v^\eps_{j-1}}{2 \Delta x} &- \frac{\lambda}{2 \Delta x} \left(  u^\eps_{j+1} -2 u^\eps_j + u^\eps_{j-1} \right) \notag \\
    &- \frac 1 2 \left( s_{j+1}^- - (s_j^+ + s_{j}^-) + s_{j-1}^+\right)= 0, \\
    \partial_t v^\eps_j + \frac{\lambda^2}{2 \Delta x} \left(  u^\eps_{j+1} - u^\eps_{j-1} \right) &- \frac{\lambda}{2 \Delta x}  \left(  v^\eps_{j+1} -2 v^\eps_j + v^\eps_{j-1} \right) \notag \\ &+ \frac \lambda 2 \left( s_{j+1}^-  + s_j^+ - s_{j}^- - s_{j-1}^+\right)= \frac{1}{\eps} \left( f(u^\eps_j) - v^\eps_j \right).
  \end{align}
\end{subequations}
The resulting scheme adds second order extension terms to its first order version, to which we refer in the following by $\mathcal{S}^u_j$ and $\mathcal{S}^v_j$. The included limited slopes \eqref{eq:mclimiter} can be alternatively computed by applying the minmod operator \eqref{eq:minmod} to the terms
\begin{equation}\label{eq:mclimiteroriginal}
 \frac{v_j^\eps - v_{j-1}^\eps  \mp \lambda (u_j^\eps - u_{j-1}^\eps)}{\Delta x},  \frac{v_{j+1}^\eps - v_{j-1}^\eps  \mp \lambda (u_{j+1}^\eps - u_{j-1}^\eps)}{4\,\Delta x}, ~ \frac{v_{j+1}^\eps - v_{j}^\eps  \mp \lambda (u_{j+1}^\eps - u_{j}^\eps)}{\Delta x}.
\end{equation}

Analogously to Section \ref{sec:relscheme11}, we apply scheme \eqref{eq:relaxationMUSCLupwindscheme} to the coupled relaxation system \eqref{eq:relsyst11}, see also Figure \ref{fig:one-to-one-num}. In the coupled scheme (compare \eqref{eq:relaxationupwindschemeleft} and \eqref{eq:relaxationupwindschemeright}) the following second order extensions appear in the evolution formulas of the volumes next to the coupling node
\begin{subequations}
  \begin{align}
    \mathcal{S}^u_{-1} = - \frac 1 2 \left( s_{R}^- - (s_{-1}^+ + s_{-1}^-) + s_{-2}^+\right), \quad  \mathcal{S}^v_{-1}= \frac \lambda 2 \left( s_{R}^-  + s_{-1}^+ - s_{-1}^- - s_{-2}^+\right),\\
    \mathcal{S}^u_{0} = - \frac 1 2 \left( s_{1}^- - (s_{0}^+ + s_{0}^-) + s_{L}^+\right), \quad  \mathcal{S}^v_{0}= \frac \lambda 2 \left( s_{1}^-  + s_{0}^+ - s_{0}^- - s_{L}^+\right).
  \end{align}
\end{subequations}
Our coupling approach gives rise to Dirichlet boundary problems on the edges of the network and does not provide any information about the slope beyond the coupling node. Being derived from the Riemann problem coupling data is assumed spatially constant and thus, it is natural to set 
\begin{equation}
  s_R^- = s_L^+ = 0.
\end{equation}
We note that by setting $s_j^\mp =0$ for any $j\in \Z$ in the MUSCL scheme we locally recover the first order scheme. Higher order approximations beyond the coupling node are not considered in this work but have been achieved by transforming spatial to temporal information using an ADER approach as e.g. in \cite{MR3227276,Banda2016}. 
Moreover, due to \eqref{eq:laxcurves} the coupling data \eqref{eq:relaxationcouplingdata} satisfies
\begin{equation}\label{eq:characteristiccouplingdata}
  w^{\eps+}_{-1} = \frac{v_{-1}^{\eps} + \lambda_1 u_{-1}^{\eps}}{2} = \frac{v_R^{\eps} + \lambda_1 u_R^{\eps}}{2} \eqcolon w^{\eps+}_{R} \quad\text{and}\quad   w^{\eps-}_{0} = \frac{v_0^{\eps} - \lambda_2 u_0^{\eps}}{2} = \frac{v_L^{\eps} - \lambda_2 u_L^{\eps}}{2} \eqcolon w^{\eps-}_{L}.
\end{equation}
In fact \eqref{eq:characteristiccouplingdata} is an equivalent formulation to \eqref{eq:relaxationwaves} in characteristic variables. Consequently, when using coupling data $u_R^{\eps}$, $v_R^{\eps}$, $u_L^{\eps}$ and $v_L^{\eps}$  as ghost cell data at the coupling node when computing the linear reconstructions, we get due to \eqref{eq:mclimiter} and \eqref{eq:minmod}
\begin{equation}
  s_{-1}^+ = s_0^- = 0.
\end{equation}

Following the steps in Sections \ref{sec:imex} and \ref{sec:limit} we discretize in time and take the relaxation limit. We thus obtain the second order limit scheme
  \begin{subequations}\label{eq:limitschemehighorder}
    \begin{align}
      u^{n+1}_j &= u^{n}_j - \frac{\Delta t}{2 \Delta x}\left( f_{\alpha(j)}(u^{n}_{j+1}) - f_{\alpha(j)}(u^{n}_{j-1})\right) + \frac{\lambda_{\alpha(j)} \Delta t}{2 \Delta x} \left( u^{n}_{j+1} - 2 u^{n}_{j} + u^{n}_{j-1}\right) \notag \\ &\quad + \frac {\Delta t} 2 \left(  s_{j+1}^{n-} - (s_j^{n+} + s_{j}^{n-}) + s_{j-1}^{n+} \right) ,\label{eq:limitschemehog}\\
      u^{n+1}_{-1} &= u^{n}_{-1} - \frac{\Delta t}{2 \Delta x}\left( v^{n}_{R} - f_1(u^{n}_{-2})\right) + \frac{\lambda_1 \Delta t}{2 \Delta x} \left( u^{n}_{R} - 2 u^{n}_{-1} + u^{n}_{-2}\right) - \frac {\Delta t} 2 \left(  s_{-1}^{n-} - s_{-2}^{n+} \right), \label{eq:limitschemehol}\\
      u^{n+1}_{0} &= u^{n}_{0} - \frac{\Delta t}{2 \Delta x}\left( f_2(u^n_{1}) - v^{n}_{L}\right) + \frac{\lambda_2 \Delta t}{2 \Delta x} \left( u^{n}_{1} - 2 u^{n}_{0} + u^{n}_{L}\right)  + \frac {\Delta t} 2 \left(  s_{1}^{n-} - s_0^{n+} \right)\label{eq:limitschemehor}
      \end{align}
    \end{subequations}
    for $j\in \Z \setminus\{-1, 0\}$ in \eqref{eq:limitschemeg} and with $\alpha(j) = 1$ for negative $j$ and $\alpha(j) = 2$ for positive $j$. The time discrete slopes $s^{n\mp}_j$ in \eqref{eq:limitschemehighorder} are obtained by applying the limiter \eqref{eq:mclimiter} to the time discrete characteristic variables
\begin{equation}
  w^{n\mp}_j = \frac 1 2 f_{\alpha(j)}(u_j^n) \mp \frac{\lambda_{\alpha(j)}}{2}u_j^n.
\end{equation}
By supplementing the coupling data \eqref{eq:limitschemers}, \eqref{eq:relaxationcouplingdata} we again obtain the conservative form \eqref{eq:conservativeform} using the numerical fluxes
\begin{equation}\label{eq:MUSCLflux}
  F_{j-1/2}^n = \begin{cases}
                 \frac 1 2 \, (f_1(u_{j-1}^n) + f_1(u_{j}^n))  - \frac {\lambda_1} 2 (u_j^n-u_{j-1}^n)  - \frac {\Delta x} 2 (s_j^{n-}-s_{j-1}^{n+}) &\text{if }j<0, \\[5pt]
                 \frac 1 {\lambda_1 + \lambda_2} \, ( \lambda_1 f_1(u_{-1}^n) + \lambda_2 f_2(u_{0}^n))  - \frac 1  {\lambda_1 + \lambda_2} ( \lambda_2^2\, u_0^n- \lambda_1^2\, u_{-1}^n) &\text{if }j=0, \\[5pt]
                 \frac 1 2 \, ( f_2(u_{j-1}^n) + f_2(u_{j}^n))  - \frac {\lambda_2} 2 (u_j^n-u_{j-1}^n)- \frac {\Delta x} 2 (s_j^{n-}-s_{j-1}^{n+}) &\text{if }j>0. 
               \end{cases}
\end{equation}
At the coupling node the numerical flux is identical to the one of the first order scheme given by \eqref{eq:conservativeform} and \eqref{eq:numfluxgeneralspeeds}. In the scheme coupling data  \eqref{eq:limitschemers}, \eqref{eq:relaxationcouplingdata} is required to compute $s_{-1}^{n-}$ and $s_0^{n+}$ according to \eqref{eq:mclimiter}, which takes into account $w^{n-}_{R} \coloneq v_R^n/2 - \lambda_1 u_R^n/2$ and $w^{n+}_{L} \coloneq v_L^n/2 + \lambda_1 u_L^n/2$. Alternatively, these slopes can be set to zero for simplicity.  

Similar to Proposition \ref{prop:consistencylimit} consistency to a second order scheme for \eqref{eq:scalarcl} on the real line is given if $f_1=f_2$, $\lambda_1=\lambda_2$ and when neglecting the numerical flux at the coupling node. For this reason the following result is deduced.

\begin{Prop}\label{prop:TVD} Let $f_1=f_2=f$, $\lambda_1=\lambda_2=\lambda$, $\lambda \Delta t \leq \Delta x/2$ and assume that the subcharacteristic condition \eqref{eq:subcharacteristic} holds. Then the scheme \eqref{eq:limitschemehighorder} using coupling data \eqref{eq:limitschemers}, \eqref{eq:relaxationcouplingdata} and $s_{-1}^{n-}=s_0^{n+}=0$ for all $n\in\N_0$ is consistent, \emph{total variation diminishing} (see e.g., \cite{godlewskiNumericalApproximationHyperbolic1996}) and converges to the weak solution of the conservation law \eqref{eq:scalarcl}.
\end{Prop}
\begin{proof}
  From a result in \cite{jinRelaxationScheme1995} follows that the scheme is total variation diminishing. In this work the authors considered a scheme of the form \eqref{eq:limitschemehog} (with uniform $\lambda$ and $f$) for $j\in\Z$ but assumed different slopes. Yet, as the proof only requires that the slopes can be written as
  \begin{equation*}
    s_j = \frac{w_{j+1}-w_j}{\Delta x} \phi_j, \quad \text{where}\quad 0\leq \phi_j \leq 2 \quad \text{and}\quad 0\leq \phi_j \, \frac{w_{j+1} - w_j}{w_j - w_{j-1}}\leq 2,
  \end{equation*}
  which is satisfied by \eqref{eq:mclimiter}, \eqref{eq:minmod} and
  \begin{equation*}
    \phi_j= \operatorname{minmod}\left(2 \frac{w_j-w_{j-1}}{w_{j+1}-w_j},~\frac{w_{j+1}-w_{j-1}}{2(w_{j+1}-w_j)},~ 2\right),
  \end{equation*}
 the result transfers to our second order scheme. We neglected time and sign indices above for clarity. Since the numerical fluxes are further consistent, the convergence result follows due to \cite{levequeFiniteVolumeMethods2002}.
\end{proof}
\section{Multiple Incoming and Outgoing Edges}\label{sec:multipleedges}
In this section, we generalize the schemes introduced in Sections \ref{sec:limit} and \ref{sec:highorder} to scalar conservation laws at nodes with $N^-$ incoming and $N^+
$ outgoing edges. To this end we consider the relaxation system introduced in Section \ref{sec:relaxation} in a network setting to derive admissible coupling data. In Section \ref{sec:fluxdistribution} we then discuss additional conditions about the flux distribution at the junction, which are required for well posedness of coupling data. In Section \ref{sec:networkschemes} we eventually present the central schemes for the network problem \eqref{eq:scalarconservationnetwork}.

\subsection{The Relaxation System on Networks}\label{sec:relsystnet}
We generalize the approach from Section \ref{sec:relsyst11} and consider the relaxation system \eqref{eq:relsyst} on a network that allows for $N^-$ incoming and $N^+$ outgoing edges. Again, we aim to take the relaxation limit in a space discretization of the network system in order to derive a scheme for the scalar network \eqref{eq:scalarconservationnetwork}. By coupling multiple systems of the form \eqref{eq:relsyst} in the same way we coupled scalar conservation laws in \eqref{eq:scalarconservationnetwork} we obtain the system 
\begin{subequations}\label{eq:relsystnet}
  \begin{align}
    \ddt u^{k} + \ddx v^{k} &= 0, && \text{in }\mathcal{E}_k\times(0, \infty), &&k\in \delta^\mp, \label{eq:relsystnetu}\\
    \ddt v^{k} + \lambda_k^2 \, \ddx u^{k} &= \frac 1 \eps (f_k(u^{k}) - v^{k}), && \text{in }\mathcal{E}_k\times(0, \infty), &&k\in \delta^\mp \label{eq:relsystnetv}
  \end{align}
\end{subequations}
governing the scalar variables $u^k$ and $v^k$, where $k\in\delta^\mp=\delta^- \cup \delta^+=\{1,\dots,N\}$ indicates the corresponding edge of the network. The system includes a stiff source term that contains the flux functions $f_1,\dots,f_{N}:$ $\R \rightarrow \R$ and the relaxation rate $\eps$ that is chosen uniform over the edges of the network. Compared to the 1-to-1 network \eqref{eq:relsyst11} we do not account for $\eps$ in the variable names here for brevity of presentation. On each edge of the network a relaxation speed $\lambda_k>0$ is given, which satisfies the subcharacteristic condition, i.e.,
\begin{equation}\label{eq:subcharacteristicnetwork}
 -\lambda_k\leq f_k^\prime(u^k) \leq \lambda_k \quad \text{for all }u^k \text{ and }k\in \delta^\mp.
\end{equation}

The smooth and compactly supported scalar functions $u^{k,0}$ determine the initial condition of system \eqref{eq:relsystnet} by
\begin{equation}\label{eq:relsystnetinit}
u^{k}(x,0)=u^{k,0}(x), \quad  v^{k}(x,0)=f_k(u^{k, 0}(x))\quad k\in \delta^\mp.
\end{equation}
Furthermore, coupling conditions for the system variables at the coupling node have the form
\begin{equation}\label{eq:relsystnetcoupling}
  \Psi\left[(u^{1}(0^-, t), v^{1}(0^-, t)), \dots, (u^{N}(0^+,t),v^{N}(0^+, t))\right]=0 \quad \text{for a.e. }t>0.
\end{equation}
Since the system is linear, $N$ conditions are necessary to obtain a well-posed problem and thus $\Psi:\R^{2N}\rightarrow \R^N$ (compare system \eqref{eq:relsyst11}, which can be written as \eqref{eq:relsystnet} with $N=2$ and requires the two conditions given in \eqref{eq:relsyst11couplingconditions}). At a fixed time $t>0$ we assume given traces $u_0^{1}, v_0^{1}, \dots, u_0^{N}, v_0^{N}$ at $x=0$. Unlike in Section \ref{sec:relsyst11} we do not indicate if a trace is incoming or outgoing by a sign index as this can be seen from the edge index. We formally define a Riemann solver for system \eqref{eq:relsystnet} by
\begin{equation}\label{eq:riemannsolverrelsystnet}
\mathcal{RS}_\text{rel}: (u_0^{1}, v_0^{1}, \dots, u_0^{N}, v_0^{N}) \mapsto (u_R^{1}, v_R^{1}, \dots, u_L^{N}, u_L^{N}).
\end{equation}
The construction of \eqref{eq:riemannsolverrelsystnet} is discussed in the remainder of this section and in Section \ref{sec:fluxdistribution}. We aim to derive admissible coupling data that verifies the conservation of the system variables in the coupling node. To ensure admissible boundary data, $(u_R^k, v_R^k)$ needs to connect to $(u_0^{k}, v_0^{k})$ by a wave with negative velocity for all $k\in \delta^-$ and  $(u_L^{k}, u_L^{k})$ needs to connect to  $(u_0^{k}, u_0^{k})$ by a wave with positive velocity for all $k\in \delta^+$. Thus, by the analysis in Appendix \ref{app:eigenvalues}, we get the conditions
\begin{equation}\label{eq:relaxationnetwaves}
  (u_R^{k}, v_R^{k}) \in L^{1-}_{\lambda_k}(u_0^{k}, v_0^{k}) \quad \text{for }k\in \delta^- \quad \text{and}\quad  (u_L^{k}, v_L^{k}) \in L^{2+}_{\lambda_k}(u_0^{k}, v_0^{k})\quad \text{for }k\in \delta^+.
\end{equation}
Since all Lax-curves are parameterized by a single parameter, \eqref{eq:relaxationnetwaves} gives rise to $N$ unknowns, which are to be determined by the coupling conditions.
As done in \eqref{eq:relsyst11couplingconditions} for the 1-to-1 network we impose the Kirchhoff condition \eqref{eq:kirchhoff} for both system variables and obtain the two conditions
\begin{subequations}\label{eq:relsystnetcouplingconditions}
  \begin{align}
\Psi_1\left[(u_R^{1}, v_R^{1}), \dots, (u_L^{N}, v_L^{N})\right] &= \sum_{j\in\delta^{-}} v_L^{j} - \sum_{k\in\delta^{+}} v_R^{k} = 0,\label{eq:kirchhoffnet}\\
\Psi_2\left[(u_R^{1}, v_R^{1}), \dots, (u_L^{N}, v_L^{N})\right] &= \sum_{j\in\delta^{-}} \lambda_j^2\, v_L^{j} - \sum_{k\in\delta^{+}} \lambda_k^2 \, v_R^{k}  = 0. 
\end{align}
\end{subequations}

\subsection{Additional Conditions in the General Case}\label{sec:fluxdistribution}
The two conditions given by \eqref{eq:relsystnetcouplingconditions} are not sufficient to obtain a well-defined Riemann solver for the network system \eqref{eq:relsystnet}. Therefore $N-2$ suitable additional conditions need to be imposed. Here we consider algebraic conditions on the auxiliary variable at the coupling node of the form
\begin{equation}\label{eq:additionalnetworkconditions}
  \Psi_{\ell+2} \left[(u_R^{1}, v_R^{1}), \dots, (u_L^{N}, v_L^{N})\right]  = -r^\ell + \sum_{j\in\delta^{-}}  \beta_j^\ell\, v_L^{j} + \sum_{k\in\delta^{+}}  \beta_k^\ell\, v_R^{k} = 0 \quad \text{for }\ell=1,\dots,N-2
\end{equation}
with parameters $\beta_k^\ell$, $r^\ell\in\R$ for $k\in\delta^\pm$ and $\ell=1,\dots,N-2$. In the relaxation limit, the variable $v$ is the flux of the state variable of the conservation law in the relaxation limit and \eqref{eq:additionalnetworkconditions} can be understood as conditions on the fluxes at the coupling node when the limit scheme is considered.

Taking into account \eqref{eq:relaxationnetwaves} we denote by $\sigma_k$ the position of $(u_R^k, v_R^k)$ for $k\in\delta^-$ or of $(u_L^k, v_L^k)$ for $k \in \delta^-$ on the corresponding Lax curve given by \eqref{eq:laxcurves}. Then we get from \eqref{eq:relsystnetcouplingconditions} and \eqref{eq:additionalnetworkconditions} the linear system
\begin{subequations}\label{eq:netsigmaconditions}
\begin{align}
  \sum_{k\in\delta^\mp} \nu_k \, \lambda_k\, \sigma_k &= -\sum_{k\in\delta^\mp} \nu_k \, v_0^k,\\
  \sum_{k\in\delta^\mp}  \lambda_k^2 \, \sigma_k &= -\sum_{k\in\delta^\mp} \nu_k \, \lambda_k^2\, u_0^k,\\
    \sum_{k\in\delta^\mp} \beta_k^\ell \,  \lambda_k\, \sigma_k &= r^\ell -\sum_{k\in\delta^\mp} \beta_k^\ell \, v_0^k\quad\text{for }\ell=1,\dots,N-2 \label{eq:netsigmaadditional}
\end{align}
\end{subequations}
where $\nu_k$ denotes the sign of the edge defined as
\begin{equation}\label{eq:edgesign}
  \nu_k=\begin{cases}
          -1 &\text{if }k\in \delta^-,\\
          1 &\text{if } k \in \delta^+.
          \end{cases}
\end{equation}
When assigning parameters for \eqref{eq:additionalnetworkconditions} care must be taken that the system \eqref{eq:netsigmaconditions} has full rank. For example in case of a 2-to-1 network and $\lambda_1=\lambda_2=\lambda_3$ this is satisfied whenever $\beta_1^1 \neq \beta_2^1$. We introduce the following vector notations for traces and coupling data
\begin{equation}
  \mathbf{u}_0 = (u_0^1, \cdots, u_0^N)^T, \quad \mathbf{v}_0 = (v_0^1, \cdots, v_0^N)^T,\quad 
  \mathbf{u}_c = (u_R^1, \cdots, u_L^N)^T, \quad \mathbf{v}_c = (v_R^1, \cdots, v_L^N)^T.
\end{equation}
Moreover, we introduce the $N \times N$ diagonal matrices $\Lambda$ and $\mathcal{N}$ with diagonal entries $\lambda_1,\dots,\lambda_N$ and $\nu_1,\dots, \nu_N$, respectively. This allows us to express the coupling data by means of the linear systems 
\begin{equation}\label{eq:relaxationnetcouplingdata}
\mathcal{A} \mathcal{N} (\mathbf u_c - \mathbf u_0)=  \mathbf b, \qquad \mathcal{A} \Lambda^{-1} (\mathbf v_c - \mathbf v_0) =  \mathbf b,
\end{equation}
which hence define \eqref{eq:riemannsolverrelsystnet}. Here $\mathcal{A}$ and $\mathbf b$ denote the system matrix and the right-hand side of \eqref{eq:netsigmaconditions}, i.e.,
\begin{equation}\label{eq:riemannmatrix}
  \mathcal{A} =
  \begin{pmatrix}
    \nu_1 \, \lambda_1 &  \cdots & \nu_{N}\, \lambda_{N} \\[5pt]
    \lambda_1^2 & \cdots & \lambda_{N}^2 \\[5pt]
    \beta_{1}^1\, \lambda_1  & \cdots &  \beta_{N}^1\, \lambda_{N} \\[5pt]
    \vdots & & \vdots  \\
    \beta_{1}^{N-2}\, \lambda_1  & \cdots &  \beta_{N}^{N-2}\, \lambda_{N}
  \end{pmatrix},
  \quad \mathbf b =
  \begin{pmatrix}
    -\sum_{k\in\delta^\mp} \nu_k \, v_0^k \\[5pt]
    -\sum_{k\in\delta^\mp} \nu_k \, \lambda_k^2\, u_0^k\\[5pt]
    r^1 -\sum_{k\in\delta^\mp} \beta_k^1 \, v_0^k \\[5pt]
    \vdots \\
     r^{N-2} -\sum_{k\in\delta^\mp} \beta_k^{N-2} \, v_0^k
  \end{pmatrix}.
\end{equation}
In the rest of this section we give general suggestions for the parameters in \eqref{eq:additionalnetworkconditions}. However, we note that they are an issue of modeling and can be freely adapted to the problem at hand. We will assume non-negativity of the flux functions $f_1, \dots, f_N$ and the trace data $v_0^1,\dots, v_0^N$. In the relaxation limit, the latter follows directly from the non-negativity of the fluxes. 

\subsubsection{Incoming edges}
In case of multiple incoming edges a reasonable assumption is that the relation between the incoming flux from a given edge to the total incoming flux in the coupling node is inherited from the traces. Since the variable $v$ determines the flux of the primary variable $u$ this is stated as
\begin{equation}\label{eq:incomingfluxdistribution}
  \frac{v_R^m}{\sum_{k\in\delta^-} v_R^k} =   \frac{v_0^m}{\sum_{k\in\delta^-} v_0^k} \quad \text{for all }m\in \delta^-,
\end{equation}
provided that both $\sum_{k\in\delta^-} v_0^k \neq 0$ and  $\sum_{k\in\delta^-} v_R^k \neq 0$. This can be rewritten as the following conditions on the coupling data  
\begin{equation}
\left(\sum_{k\in\delta^-\setminus\{m\}} v_0^k \right) v_R^m - \sum_{k\in\delta^-\setminus\{m\}} v_0^m \, v_R^k = 0 \quad\text{for all }m\in \delta^-. 
\end{equation}
It is sufficient to impose \eqref{eq:incomingfluxdistribution} for all but one incoming edges, then still, by a summation argument, it follows for all incoming edges. Thus to account for assumption \eqref{eq:incomingfluxdistribution}, the parameters for $\ell=1,\dots,N^--1$ in \eqref{eq:netsigmaadditional} can be chosen as $\beta_\ell^\ell = \sum_{k\in\delta^-\setminus\{\ell\}} v_0^k$, $\beta_k^\ell = -v_0^m$ for $k \in \delta^-\setminus\{\ell\}$, $\beta_k^\ell=0$ for $k \in \delta^+$ and $r^\ell=0$. These trace dependent parameters lead to a nonlinear Riemann solver \eqref{eq:riemannsolverrelsystnet}, whose evaluation only requires the assembly of \eqref{eq:riemannmatrix} and the solution of the systems \eqref{eq:relaxationnetcouplingdata}. In case  $\sum_{k\in\delta^-} v_0^k \neq 0$ the system \eqref{eq:netsigmaconditions} (ignoring $\ell\geq N^-$) has maximal rank. To handle the case  $\sum_{k\in\delta^-} v_0^k = 0$, we can regularize \eqref{eq:incomingfluxdistribution} by adding a small constant $\epsilon>0$ to both denominators. To account for this regularization in the parameters, we add $\epsilon$ to $\beta_\ell^\ell$ and set $r^\ell=\epsilon v^\ell$ for $\ell=1,\dots,N^--1$. This way the Riemann solver can only assign nonzero incoming fluxes in the edge $k=N^-$ if all corresponding incoming traces are zero.

\subsubsection{Outgoing edges}
The distribution of the fluxes among the outgoing edges in a network is often described by a distribution matrix, see e.g. \cite{garavelloTrafficFlowNetworks2006}.
This matrix is comprised of rates that describe how the flux from any incoming edge is distributed to the outgoing edges. Similarly, we impose by 
\begin{equation}\label{eq:outgoingfluxdistribution}
  \sum_{k \in \delta^-} \alpha^{m}_k v^k_R =  v^m_L
\end{equation}
that the flux that enters the coupling node from edge $k\in\delta^-$ is distributed to edge $m\in \delta^+$ with rate $\alpha^m_k$. The rates are chosen such that $0 \leq \alpha^m_k\leq 1$ and 
\begin{equation}\label{eq:ratesum}
  \sum_{m \in \delta^+} \alpha^{m}_k = 1 \quad \text{for all }k\in\delta^-.
\end{equation}
Similar to condition \eqref{eq:incomingfluxdistribution} in case of incoming edges, condition \eqref{eq:outgoingfluxdistribution} only needs to be imposed to all but one outgoing edges, then it also follows for the remaining outgoing edge due to \eqref{eq:kirchhoffnet} and \eqref{eq:ratesum}. Thus, if we set the parameters in \eqref{eq:netsigmaconditions} for all $\ell=N^-,\dots,N-2$ to $\beta_k^\ell=\alpha^\ell_k$ for $k \in \delta^-$, $\beta^\ell_{\ell+1} = - 1$, $\beta_k^\ell=0$ for $k\in\delta^+\setminus\{\ell+1\}$ and $r^\ell=0$, condition \eqref{eq:outgoingfluxdistribution} holds for $m \in \delta^+$. When we further use the suggested parameters accounting for the incoming edges as discussed above, system~\eqref{eq:netsigmaconditions} has full rank and the Riemann solver \eqref{eq:riemannsolverrelsystnet} is well-defined.  

\subsection{Central Scheme for Networks of Scalar Conservation Laws}\label{sec:networkschemes}
A scheme for the scalar network \eqref{eq:scalarconservationnetwork} is obtained by a straightforward generalization of the steps in Sections \ref{sec:semidiscrete}--\ref{sec:highorder} and taking into account the generalized Riemann solver at the coupling node introduced in Sections \ref{sec:relsystnet} and \ref{sec:fluxdistribution}. Unlike in Section \ref{sec:scheme} we do not substitute explicit formulas of coupling data into the scheme in the general network case as was done in \eqref{eq:limitschemesub}. Instead each edge in the derivation is considered separately.

Hence, after discretizing system \eqref{eq:relsystnet} in space by scheme \eqref{eq:relaxationupwindscheme}, discretizing in time as in \eqref{eq:unsplit} and taking the relaxation limit we obtain a network generalization of scheme \eqref{eq:limitscheme} that can be written in conservative form as
\begin{subequations}\label{eq:netscheme}
\begin{equation}\label{eq:conservativeformnet}
  u_j^{k,n+1} = u_j^{k,n} - \frac{\Delta t }{\Delta x} \left( F_{j+1/2}^{k,n} - F_{j-1/2}^{k,n}\right) \quad \text{for all }j \in \mathcal{I}_k,
\end{equation}
where $\mathcal{I}_k=\Z^-$ if $k\in \delta^-$ or  $\mathcal{I}_k=\Z^+_0$ if $k\in \delta^+$. The numerical fluxes are given by
\begin{equation}\label{eq:netfluxes}
  F_{j-1/2}^{k,n} = \begin{cases}
                 \frac 1 2 \, (f_k(u_{j}^{k,n}) + f_k(u_{j-1}^{k,n}))  - \frac {\lambda_k} 2 (u_j^{k,n}-u_{j-1}^{k,n}) - {\mathcal{S}}^{k,n}_{j-1/2} & \text{if }j\,\nu_k>0,\\[5pt]
                 \frac 1 2 \, ( v_R^{k,n} + f_k(u_{-1}^{k,n}))  - \frac {\lambda_k} 2 (u_R^{k,n}-u_{-1}^{k,n}) &\text{if }j=0 \text{ and } k \in \delta^-, \\[5pt]
                 \frac 1 2 \, ( f_k(u_{0}^{k,n}) + v_L^{k,n})  - \frac {\lambda_k} 2 (u_0^{k,n}-u_{L}^{k,n}) &\text{if }j=0 \text{ and } k\in \delta^+. 
               \end{cases}
 \end{equation}
\end{subequations}
If the included high order extension terms ${\mathcal{S}}^{k,n}_{j-1/2}$ are set to zero, the first order scheme is obtained. To get the second order scheme, as derived in Section \ref{sec:highorder} for the 1-to-1 network, we set
\begin{equation}
{\mathcal{S}}^{k,n}_{j-1/2} = \frac {\Delta x} 2 (s_j^{k,n-}-s_{j-1}^{k,n+}),
\end{equation}
where the slopes of the linear reconstructions are given by
\begin{equation}\label{eq:mclimiternet}
  s_j^{k,n\mp} = \operatorname{minmod}\left( 2\frac{w^{k,n\mp}_{j}  - w^{k,n\mp}_{j-1} }{\Delta x}, ~\frac{w^{k,n\mp}_{j+1}  - w^{k,n\mp}_{j-1} }{2\,\Delta x}, ~ 2 \frac{w^{k,n\mp}_{j+1}  - w^{k,n\mp}_{j} }{\Delta x} \right),
\end{equation}
$w^{k,n\mp}_{j} = \frac 1 2 f_k(u^{k,n}_j) \mp \frac {\lambda_k}{2} u^{k,n}_j$ and the minmod operator \eqref{eq:minmod}. Coupling data is obtained by applying the Riemann solver \eqref{eq:riemannsolverrelsystnet} and setting
  \begin{equation}\label{eq:netschemers}
    \mathcal{RS}_\text{rel} (u_{-1}^{1,n}, f_1(u_{-1}^{1,n}), \dots, u_0^{N, n}, f_N(u_0^{N, n}))  \eqcolon (u_R^{1,n}, v_R^{1,n}, \dots, u_L^{N,n}, v_L^{N,n}).
  \end{equation}
The form \eqref{eq:conservativeformnet} does not imply mass conservation at the coupling node, which we instead analyze by considering the incoming and outgoing numerical fluxes and the coupling data. Taking into account \eqref{eq:laxcurves} in \eqref{eq:netfluxes} for $j=0$, it follows that the numerical fluxes coincide with the fluxes obtained by the Riemann solver, i.e., 
\begin{equation}\label{eq:nodefluxes}
   F^{k,n}_{-1/2} = v_R^{k,n} \quad\text{for }k\in\delta^-, \quad  F^{k,n}_{-1/2} = v_L^{k,n} \quad\text{for }k\in\delta^+. 
\end{equation}
Thus the following result is a consequence of \eqref{eq:kirchhoffnet} and \eqref{eq:nodefluxes}.
\begin{Prop}[Conservation property of the central scheme]
  The central scheme for the scalar network \eqref{eq:scalarconservationnetwork} given by \eqref{eq:netscheme} and coupling data \eqref{eq:netschemers}, \eqref{eq:riemannsolverrelsystnet} is conservative in the coupling node, i.e., it holds
  \begin{equation}
    \sum_{k\in\delta^-} F_{-1/2}^{k,n} =  \sum_{k\in\delta^+}  F_{-1/2}^{k,n}\quad \text{for all }n \in \N_0.
  \end{equation}
  It is further conservative on the full network in the case $f_1(0)=f_2(0)=\dots=f_N(0)=0$ and on bounded networks, where all edges that connect to the node are bounded and zero-flux boundary conditions are imposed.
  \end{Prop}

  \section{Numerical Experiments}
  In this section we apply the derived schemes in various numerical experiments to demonstrate their capabilities and performance. We focus on the limit schemes for system \eqref{eq:scalarconservationnetwork} and do not take into account discretizations of the relaxation system. The experiments consider scalar conservation laws on 1-to-1 and 2-to-1 networks.

  We assume that the network nodes are bounded and parameterized by $(-1, 0)$ if they are incoming or by $(0, 1)$ if they are outgoing. Each edge is discretized over $m$ uniform mesh cells of size $\Delta x = 1/m$. We use a fixed relaxation speed $\lambda$ over the full network and take time increments as
  \begin{equation}\label{eq:cflcondition}
    \Delta t = \operatorname{CFL} \, \frac{\Delta x}{\lambda}.
  \end{equation}
  Details on the used number of mesh cells, Courant number and boundary conditions are provided in the individual experiment descriptions. The employed computer programs are implemented in the Julia programming language \cite{bezanson2017julia} and publicly available from~\cite{CentralNetworkSchemeCode}.
  
  \subsection{Inviscid Burgers' equation}
In the first numerical experiment we consider the inviscid Burgers' equation on a 1-to-1 network. To test the accuracy of the schemes we impose identical fluxes on both edges of the network given by \eqref{eq:scalarconservation11} and $f_1(u)=f_2(u)=1/2 \,u^2$. We consider an experiment adapted from \cite{kurganovSemidiscreteCentralupwindSchemes2001} and employ the initial data
\begin{equation}
    u_0(x) = \frac 1 2 + \frac {\sin(\pi(x+1))}{2}
\end{equation}
  and periodic boundary conditions. While the solution of this problem is smooth at small times it later develops a shock discontinuity. This can be seen in Figure \ref{fig:burgersexperiment}, which shows the numerical solutions computed by the first order scheme \eqref{eq:limitschemeconservative} and the MUSCL scheme \eqref{eq:conservativeform}, \eqref{eq:MUSCLflux}. In the computations $m=200$ mesh cell were used left and right from the coupling node, the relaxation speed was assumed $\lambda=1$ and Courant numbers were chosen $0.9$ in the first order and $0.2$ in the MUSCL scheme. The computations show that the higher order scheme achieves a sharper resolution of the developing shock. Furthermore, the reduction in accuracy in the coupling node ($x=0$), which this scheme experiences, is visible at time instance $t=0.2$.
  \begin{figure}
  \centering
  \begin{tikzpicture}
  \begin{groupplot}[
        group style={group size=3 by 2,
            horizontal sep = .5 cm, 
            vertical sep = .2 cm,
            xticklabels at=edge bottom,
            yticklabels at=edge left}, 
          width = .4 \linewidth,
          height = .28 \linewidth,
          xmin=-1,xmax=1,ymin=-.1, ymax=1.1,
          every tick label/.append style={font=\scriptsize},
          title style={font=\scriptsize},
          every axis plot/.append style={
            only marks, mark=+, mark size=1.0 pt, line width=0.15
          }
          ]
        \nextgroupplot[title={$t=0.2$}]
        \addplot [] table [x index=0, y index=1] {input/burgers_1storder_1.dat};
        \nextgroupplot[title={$t=0.4$}]
        \addplot [] table [x index=0, y index=1] {input/burgers_1storder_2.dat};
        \nextgroupplot[title={$t=0.65$}]
        \addplot [] table [x index=0, y index=1] {input/burgers_1storder_3.dat};
        \nextgroupplot[]
        \addplot [] table [x index=0, y index=1] {input/burgers_2ndorder_1.dat};
        \nextgroupplot[]
        \addplot [] table [x index=0, y index=1] {input/burgers_2ndorder_2.dat};
        \nextgroupplot[]
        \addplot [] table [x index=0, y index=1] {input/burgers_2ndorder_3.dat};
  \end{groupplot}
\end{tikzpicture}
  \caption{Numerical results of the first (top row) and second (bottom row) order central scheme applied to the inviscid Burgers' equation on a 1-to-1 network. A shock discontinuity develops as time evolves. Computations employed $200$ mesh cells left and right from the network node at $x=0$ and Courant numbers $\text{CFL}=0.9$ in the first and $\text{CFL}=0.2$ in the second order scheme. }\label{fig:burgersexperiment} 
\end{figure}
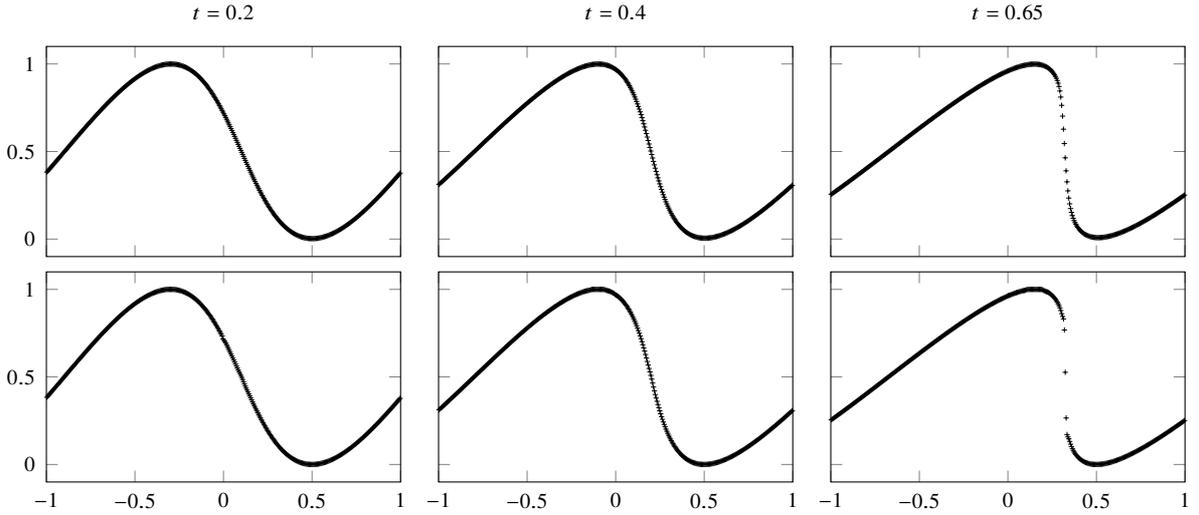

To further investigate the impact of the coupling node we computed the $L^1$ and $L^\infty$ errors after mesh refinement as well as the corresponding experimental order of convergence (EOC). Besides the first order scheme (central scheme) we distinguish three variants of the MUSCL scheme: the first (central MUSCL) makes use of the coupling data in the computation of the slopes  $s_{-1}^{n-}$ and $s_0^{n+}$, see Section \ref{sec:highorder}, the second (central MUSCL TVD) sets these slopes to zero and is TVD according to Proposition \ref{prop:TVD} and the third (uncoupled MUSCL) applies the scheme to the uncoupled case, where the inviscid Burgers' equation is considered on the domain $(-1,1)$ discretized over $2m$ cells. To see how the error depends on $\Delta x$ we computed different mesh solutions with time increments chosen according to \eqref{eq:cflcondition} and $\text{CFL}=0.49$ for the first order scheme and fixed $\Delta t=2 \times 10^{-6}$ for the MUSCL schemes. We computed the errors at time $t=0.5$ when the problem still admits a smooth solution.

\begin{table}
    \centering
    \scriptsize
    \begin{tabular}{r @{\hspace{4em}} rr @{\hspace{4em}} rr  @{\hspace{4em}} rr  @{\hspace{4em}} rr}
      \toprule
      &  \multicolumn{2}{l}{central scheme} &   \multicolumn{2}{l}{central MUSCL} &  \multicolumn{2}{l}{central MUSCL TVD} &  \multicolumn{2}{l}{uncoupled MUSCL} \\[3pt]
      $\frac{1}{\Delta x}$ & $L^1$-error & EOC & $L^1$-error & EOC & $L^1$-error & EOC & $L^1$-error & EOC\\
      \midrule
100 & $2.413 \times 10^{-2}$ & & $1.848 \times 10^{-3}$ & & $2.892 \times 10^{-3}$ & & $9.015 \times 10^{-4}$ & \\
200 & $1.339 \times 10^{-2}$ & 0.85 & $5.009 \times 10^{-4}$ & 1.88 & $7.914 \times 10^{-4}$ & 1.87 & $1.863 \times 10^{-4}$ & 2.27 \\
400 & $7.044 \times 10^{-3}$ & 0.93 & $1.272 \times 10^{-4}$ & 1.98 & $2.017 \times 10^{-4}$ & 1.97 & $4.838 \times 10^{-5}$ & 1.94 \\
800 & $3.639 \times 10^{-3}$ & 0.95 & $3.137 \times 10^{-5}$ & 2.02 & $4.949 \times 10^{-5}$ & 2.03 & $1.275 \times 10^{-5}$ & 1.92 \\
\bottomrule                                                                            
  \end{tabular}
  \caption{$L^1$ errors and EOCs in space in the inviscid Burgers' equation experiment. Errors were considered at time instance $t=0.5$ and EOCs were computed by the formula $\log2(E_1/E_2)$, where $E_1$ and $E_2$ denote the errors of a scheme in two consecutive lines of the table. The computed EOCs confirm the expected orders of the schemes.}\label{tab:l1}
\end{table}

\begin{table}
  \centering
  \scriptsize
  \begin{tabular}{r @{\hspace{4em}} rr @{\hspace{4em}} rr  @{\hspace{4em}} rr  @{\hspace{4em}} rr}
    \toprule
    &  \multicolumn{2}{l}{central scheme} &   \multicolumn{2}{l}{central MUSCL} &  \multicolumn{2}{l}{central MUSCL TVD} &  \multicolumn{2}{l}{uncoupled MUSCL} \\[3pt]
    $\frac{1}{\Delta x}$ & $L^\infty$-error & EOC & $L^\infty$-error & EOC & $L^\infty$-error & EOC & $L^\infty$-error & EOC\\
    \midrule
100 & $4.626 \times 10^{-2}$ & & $1.677 \times 10^{-2}$ & & $2.419 \times 10^{-2}$ & & $7.262 \times 10^{-3}$ & \\
200 & $2.881 \times 10^{-2}$ & 0.68 & $5.614 \times 10^{-3}$ & 1.58 & $8.769 \times 10^{-3}$ & 1.46 & $1.696 \times 10^{-3}$ & 2.10 \\
400 & $1.719 \times 10^{-2}$ & 0.75 & $2.096 \times 10^{-3}$ & 1.42 & $3.397 \times 10^{-3}$ & 1.37 & $3.536 \times 10^{-4}$ & 2.26 \\
800 & $9.694 \times 10^{-3}$ & 0.83 & $8.188 \times 10^{-4}$ & 1.36 & $1.362 \times 10^{-3}$ & 1.32 & $9.623 \times 10^{-5}$ & 1.88 \\
\bottomrule
\end{tabular}
\caption{$L^\infty$ errors and EOCs in space in the inviscid Burgers' equation experiment. Errors were considered at time instance $t=0.5$ and EOCs were computed by the formula $\log2(E_1/E_2)$, where $E_1$ and $E_2$ denote the errors of a scheme in two consecutive lines of the table. EOCs in central MUSCL and central MUSCL TVD are reduced due to the order reduction at the coupling node.}\label{tab:linf}
\end{table}

The computed $L^1$ errors shown in Table \ref{tab:l1} confirm the expected orders of convergence, i.e., first order in the central scheme and second order in the MUSCL variants. The presence of the coupling node in the schemes central MUSCL and central MUSCL TVD only slightly reduces the accuracy when compared to the uncoupled MUSCL scheme and does not interfere in the convergence order in $L^1$. Thereby the central MUSCL scheme, which was also TVD in all numerical tests, yields a higher accuracy than the central MUSCL TVD scheme. The errors in $L^\infty$ shown in Table \ref{tab:linf} behave differently. While the central scheme and the uncoupled MUSCL scheme yield the expected first and second experimental order, respectively, the experimental orders of the schemes central MUSCL and central MUSCL TVD are significantly reduced due to the handling of the coupling node. Still, these two schemes achieve high accuracy in $L^\infty$ similar to the uncoupled scheme. 

\subsection{Traffic Flow}
The second numerical experiment is concerned with a traffic scenario and imposes the Lighthill-Whitham-Richards (LWR) model \cite{lighthillKinematicWavesII1955, richardsShockWavesHighway1956}
\begin{equation}\label{eq:lwr}
  \ddt u + \ddx \left( u \left(1 - \frac{u}{u_\text{max}} \right) \right) = 0
\end{equation}
on the edges of a 2-to-1 network. In more details, we consider the scalar problem \eqref{eq:scalarconservationnetwork} with two incoming edges ($N^-=2$) and one outgoing edge ($N^+=1$). We assume that the outgoing edge has larger capacity than the incoming ones and allows for higher traffic densities. This is reflected in the flux functions, which we set
\begin{equation*}
  f_1(u)=f_2(u) = u(1-u), \quad f_3(u) = u \left( 1- \frac{u}{1.2} \right).
\end{equation*}
Coupling conditions for traffic models have been a topic of high interest, see e.g., \cite{garavelloTrafficFlowNetworks2006, gottlichSecondOrderTrafficFlow2021}. We are interested how the common \emph{flow maximization} approach for coupling of \eqref{eq:lwr} on networks compares to the coupling model implied by the presented scheme.

We sketch how coupling data is obtained according to flow maximization on 2-to-1 networks. Hereby we focus on the computation of the fluxes $v^1_R= f_1(u^1_R)$, $v^2_R= f_2(u^2_R)$ and $v^3_L= f_3(u^3_L)$. To constitute admissible boundary data these fluxes need to satisfy the demand and supply conditions, which impose upper bounds of the form
\begin{equation}\label{eq:demandsupply}
  v^1_R \leq d_1(u_0^1), \quad v^2_R \leq d_2(u_0^2), \quad v^3_L \leq s_3(u_0^3)
\end{equation}
for details see \cite{garavelloTrafficFlowNetworks2006}.
Coupling fluxes are then chosen maximal under constraints given by \eqref{eq:demandsupply} and the Kirchhoff conditon \eqref{eq:kirchhoffnet}. Two cases can occur.
In the so called \emph{free flow} case, where $ d_1(u_0^1) + d_2(u_0^2) \leq s_3(u_0^3)$, we take $v^1_R = d_1(u_0^1)$, $v^2_R = d_2(u_0^2)$ and $v^3_L= d_1(u_0^1) + d_2(u_0^2)$.
The complimentary case is referred to as \emph{congestion}. Here we set $v^3_L=s_3(u_0^3)$ and employ the right of way parameter $0\leq \beta \leq 1$ to set $v_R^1=\beta s_3(u_0^3)$ and $v_R^2=(1-\beta) s_3(u_0^3)$. If this leads to a violation of \eqref{eq:demandsupply} for either $v^1_R$ or $v^2_R$ the affected coupling flux is chosen as the respective upper bound and its counterpart is computed from the Kirchhoff condition. This procedure defines a Riemann solver, which can be applied to cell averages of a numerical scheme. To account for flow maximization in numerical simulations we employed scheme \eqref{eq:netscheme} and replaced the fluxes at the coupling node by 
\begin{equation*}
  F^{1,n}_{-1/2} = v^{1,n}_R, \quad  F^{2,n}_{-1/2} = v^{2,n}_R, \quad  F^{3,n}_{-1/2} = v^{3,n}_L, 
\end{equation*}
where the coupling fluxes were computed from the trace data $u_0^{1,n}$, $u_0^{2,n}$ and $u_0^{3,n}$ as above.

In contrast, coupling data in the central scheme was obtained from the linear systems \eqref{eq:relaxationnetcouplingdata}. The parameters in system \eqref{eq:netsigmaconditions} were chosen so that \eqref{eq:incomingfluxdistribution} was satisfied, i.e., $\beta^1_1= v^2_0$, $\beta^1_2=-v_0^1$ and $\beta^1_3=r^1=0$.
In the numerical computations for both coupling models we used $m=200$ cells on each edge and assumed $\lambda=1$. We further imposed zero-flux boundary conditions on the incoming edges and homogeneous Neumann boundary conditions on the outgoing edge.

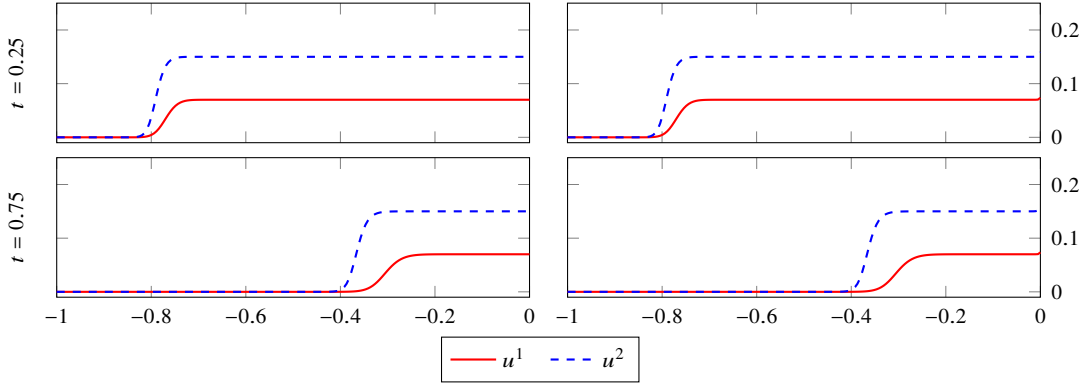
\begin{figure}
  \centering
  \begin{tikzpicture}
  \begin{groupplot}[
        group style={group size=2 by 2,
            horizontal sep = .5 cm, 
            vertical sep = .2 cm,
            xticklabels at=edge bottom,
            yticklabels at=edge right}, 
          width = .5 \linewidth,
          height = .22 \linewidth,
          xmin=-1,xmax=0,ymin=-.01, ymax=.25,
          every tick label/.append style={font=\scriptsize},
          label style={font=\scriptsize},
          legend to name=leg:free-flow,
          legend columns=-1,
          legend style={font=\scriptsize, fill=none, /tikz/every even column/.append style={column sep=.25cm}}
          ]
        \nextgroupplot[ylabel={$t=0.25$}]
        \addplot [color=red, thick] table [x index=0, y index=3] {input/LWR_TrafficFlowMaximization_freeFlow_2.dat};
        \addplot [color=blue, thick, dashed] table [x index=0, y index=4] {input/LWR_TrafficFlowMaximization_freeFlow_2.dat};
        \nextgroupplot[]
        \addplot [color=red, thick] table [x index=0, y index=3] {input/LWR_CentralRelaxationLimit_freeFlow_2.dat};
        \addplot [color=blue, thick, dashed] table [x index=0, y index=4] {input/LWR_CentralRelaxationLimit_freeFlow_2.dat};
        \nextgroupplot[ylabel={$t=0.75$}]
        \addplot [color=red, thick] table [x index=0, y index=3] {input/LWR_TrafficFlowMaximization_freeFlow_3.dat};
        \addplot [color=blue, thick, dashed] table [x index=0, y index=4] {input/LWR_TrafficFlowMaximization_freeFlow_3.dat};
        \nextgroupplot[]
        \addplot [color=red, thick] table [x index=0, y index=3] {input/LWR_CentralRelaxationLimit_freeFlow_3.dat};
        \addplot [color=blue, thick, dashed] table [x index=0, y index=4] {input/LWR_CentralRelaxationLimit_freeFlow_3.dat};
        \legend{$u^1$, $u^2$}
  \end{groupplot}
\end{tikzpicture}
\ref{leg:free-flow}
  \caption{Incoming edges in the 2-to-1 LWR network in case of free flow. Flow maximization (left) and the central approach (right) at the coupling node are compared. The first order central scheme with $\text{CFL}=0.49$ and $m=200$ cells per edge was used for the numerical simulation, nodal fluxes were replaced in case of flow maximization. The central approach yields the same dynamics of the solution as flow maximization.}\label{fig:freeflow} 
\end{figure}

To investigate the case of free flow, we first consider a numerical experiment with constant initial data on the edges chosen as $u^{0,1}= 0.07$, $u^{0,2}= 0.15$ and $u^{0,3}= 0.2$. In both coupling models the traffic freely propagates from the incoming edges to the outgoing edge and eventually out of the network. The numerical solutions computed by the first order central scheme are presented in Figure \ref{fig:freeflow} and show that the central scheme reproduces the dynamics of flow maximization in this case.

\begin{figure}
  \centering
  \begin{tikzpicture}
  \begin{groupplot}[
        group style={group size=2 by 2,
            horizontal sep = .5 cm, 
            vertical sep = .2 cm,
            xticklabels at=edge bottom,
            yticklabels at=edge right}, 
          width = .5 \linewidth,
          height = .22 \linewidth,
          xmin=-1,xmax=0,ymin=-.01, ymax=1.1,
          every tick label/.append style={font=\scriptsize},
          legend to name=leg:congestion,
          legend columns=-1,
          legend style={font=\scriptsize, fill=none, /tikz/every even column/.append style={column sep=.25cm}},
          label style={font=\scriptsize}
          ]
        \nextgroupplot[ylabel={$t=0.5$}]
        \addplot [color=red, thick] table [x index=0, y index=3] {input/LWR_TrafficFlowMaximization_congestion_beta2_3.dat};
        \addplot [color=blue, thick, dashed] table [x index=0, y index=4] {input/LWR_TrafficFlowMaximization_congestion_beta2_3.dat};
        \addplot [color=red, thick, dotted] table [x index=0, y index=3] {input/LWR_TrafficFlowMaximization_congestion_beta5_3.dat};
        \addplot [color=blue, thick, dotted] table [x index=0, y index=4] {input/LWR_TrafficFlowMaximization_congestion_beta5_3.dat};
        \nextgroupplot[]
        \addplot [color=red, thick] table [x index=0, y index=3] {input/LWR_CentralRelaxationlimit_congestion_3.dat};
        \addplot [color=blue, thick, dashed] table [x index=0, y index=4] {input/LWR_CentralRelaxationlimit_congestion_3.dat};
        \nextgroupplot[ylabel={$t=1$}]
        \addplot [color=red, thick] table [x index=0, y index=3] {input/LWR_TrafficFlowMaximization_congestion_beta2_4.dat};
        \addplot [color=blue, thick, dashed] table [x index=0, y index=4] {input/LWR_TrafficFlowMaximization_congestion_beta2_4.dat};
        \addplot [color=red, thick, dotted] table [x index=0, y index=3] {input/LWR_TrafficFlowMaximization_congestion_beta5_4.dat};
        \addplot [color=blue, thick, dotted] table [x index=0, y index=4] {input/LWR_TrafficFlowMaximization_congestion_beta5_4.dat};
        \nextgroupplot[]
        \addplot [color=red, thick] table [x index=0, y index=3] {input/LWR_CentralRelaxationlimit_congestion_4.dat};
        \addplot [color=blue, thick, dashed] table [x index=0, y index=4] {input/LWR_CentralRelaxationlimit_congestion_4.dat};
        \legend{$u^1$, $u^2$}
      \end{groupplot}
    \end{tikzpicture}
\ref{leg:congestion}
  \caption{Incoming edges in the 2-to-1 LWR network in case of congestion. Flow maximization (left) and the central approach (right) at the coupling node are compared. The first order central scheme with $\text{CFL}=0.2$ and $m=200$ cells per edge was used for the numerical simulation, nodal fluxes were replaced in case of flow maximization. Solutions in the flow maximization case are shown for $\beta=0.2$ (solid red and dashed blue line) and $\beta = 0.5$ (dotted lines). The central approach leads to the same qualitative dynamics as flow maximization in case of $\beta=0.2$.}\label{fig:congestionincoming} 
\end{figure}
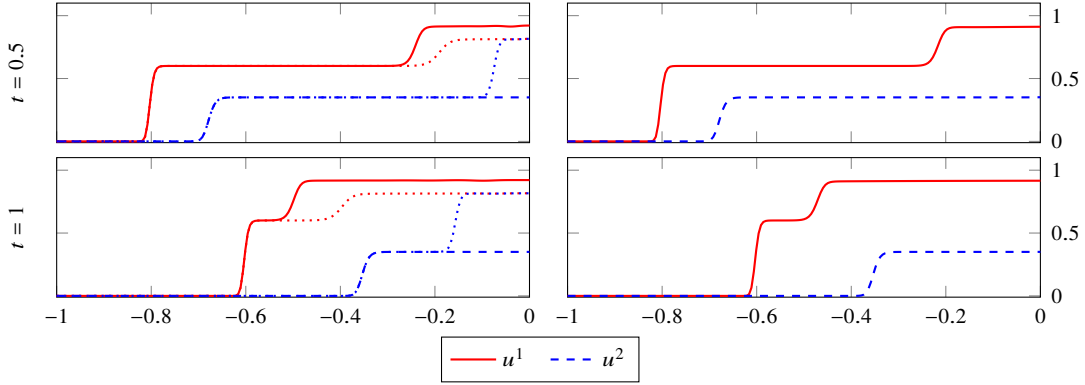

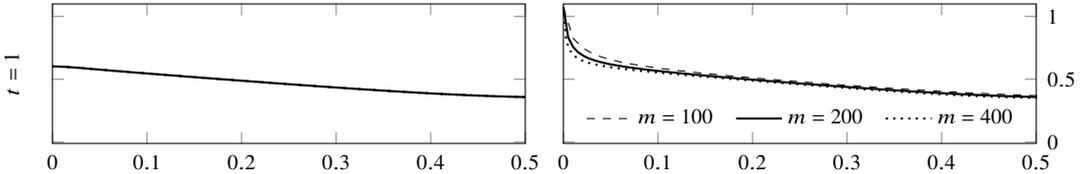
\begin{figure}
  \centering
  \begin{tikzpicture}
  \begin{groupplot}[
        group style={group size=2 by 1,
            horizontal sep = .5 cm, 
            vertical sep = .2 cm,
            xticklabels at=edge bottom,
            yticklabels at=edge right}, 
          width = .5 \linewidth,
          height = .22 \linewidth,
          xmin=0,xmax=.5,ymin=-.01, ymax=1.1,
          legend columns=-1,
          legend style={font=\scriptsize, draw=none, fill=none, /tikz/every even column/.append style={column sep=.22cm}},
          legend pos= south west,
          every tick label/.append style={font=\scriptsize},
          label style={font=\scriptsize}
          ]
        \nextgroupplot[ylabel={$t=1$}]
        \addplot [color=black, thick] table [x index=2, y index=5] {input/LWR_TrafficFlowMaximization_congestion_beta2_4.dat};
        \addplot [color=black, thick, dotted] table [x index=2, y index=5] {input/LWR_TrafficFlowMaximization_congestion_beta5_4.dat};
        \nextgroupplot[]
        \addplot [color=black, dashed] table [x index=2, y index=5] {input/LWR_CentralRelaxationlimit_congestion_refine100_4.dat};
        \addplot [color=black, thick] table [x index=2, y index=5] {input/LWR_CentralRelaxationlimit_congestion_4.dat};
        \addplot [color=black, thick, dotted] table [x index=2, y index=5] {input/LWR_CentralRelaxationlimit_congestion_refine400_4.dat};
        \legend{$m=100$, $m=200$, $m=400$},
      \end{groupplot}
\end{tikzpicture}
  \caption{Outgoing edge in the 2-to-1 LWR network in case of congestion.  Flow maximization (left) and the central approach (right) at the coupling node are compared. The first order central scheme employing $\text{CFL}=0.2$ and $m=200$ cells per edge was used for the numerical simulation, nodal fluxes were replaced in case of flow maximization. The central approach is shown for various mesh resolutions are shown. The central approach introduces a layer next to the coupling node, which decays in width as the mesh is refined.}\label{fig:congestionoutgoing} 
\end{figure}

\begin{figure}
  \centering
  \begin{tikzpicture}
  \begin{groupplot}[
        group style={group size=2 by 1,
            horizontal sep = .5 cm, 
            vertical sep = .2 cm,
            xticklabels at=edge bottom,
            yticklabels at=edge right}, 
          width = .5 \linewidth,
          height = .22 \linewidth,
          xmin=-1,xmax=0,ymin=-.01, ymax=1.1,
          every tick label/.append style={font=\scriptsize},
          label style={font=\scriptsize},
          legend to name=leg:congestion-second-order,
          legend columns=-1,
          legend style={font=\scriptsize, fill=none, /tikz/every even column/.append style={column sep=.25cm}}
          ]
        \nextgroupplot[ylabel={$t=1$}]
        \addplot [color=red, thick] table [x index=0, y index=3] {input/LWR_CentralRelaxationLimit_congestion_SecondOrder_4.dat};
        \addplot [color=blue, thick, dashed] table [x index=0, y index=4] {input/LWR_CentralRelaxationLimit_congestion_SecondOrder_4.dat};
        \addplot [color=black, thick] coordinates {(-5,0)};
        \nextgroupplot[xmin=0,xmax=0.5]
        \addplot [color=red, thick] coordinates {(-5,0)};
        \addplot [color=blue, thick, dashed] coordinates {(-5,0)};
        \addplot [color=black, thick] table [x index=2, y index=5] {input/LWR_CentralRelaxationlimit_congestion_SecondOrder_4.dat};
        \legend{$u^1$, $u^2$, $u^3$}
      \end{groupplot}
    \end{tikzpicture}
    \ref{leg:congestion-second-order}
  \caption{Incoming (left) and outgoing (right) edges in the 2-to-1 LWR network in case of congestion. The central MUSCL scheme with $\text{CFL}=0.2$ and $m=200$ cells per edge was used for the numerical simulation.}\label{fig:congestionhighorder} 
\end{figure}
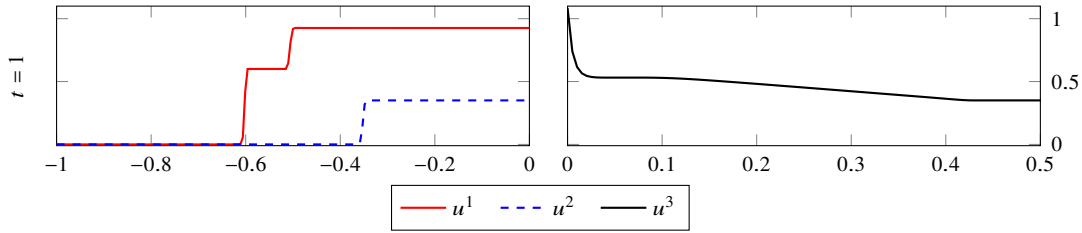

In a second numerical experiment we consider another Riemann problem using the modified initial data
$u^{0,1}= 0.6$, $u^{0,2}= 0.35$ and $u^{0,3}= 0.35$. The larger traffic densities lead to congestion at the coupling node. The numerical results by the first order scheme depicted in Figure~\ref{fig:congestionincoming} exhibit backward moving waves in the first incoming edge for both the flow maximization and the central approach. The behavior of the solution in the flow maximization case depends on the right of way parameter $\beta$. While smaller $\beta$ (solution for $\beta=0.2$ shown) lead to backward moving waves only in the first incoming edge, larger $\beta$ (solution for $\beta=0.5$ shown) lead to waves in both incoming edges. The solution of the central approach on the incoming edges is qualitatively similar to the case of flow maximization and a small right of way parameter. On the outgoing edge the central approach introduces a layer next to the coupling node connecting the sum of both incoming traffic densities near the coupling node to a decaying profile that is also obtained by the flow maximization approach, see Figure~\ref{fig:congestionoutgoing}. Computations on various meshes reveal that the layer is mesh dependent and decreases as the mesh is refined. Moreover, Figure~\ref{fig:congestionhighorder} shows that the central MUSCL scheme recovers the same dynamics as the first order central scheme but yields a smaller layer at the coupling node and higher resolution of the discontinuities.

\subsection{Coupled Two-Phase Flow Model}
Lastly, we apply our approach to the Buckley--Leverett equation \cite{buckleyMechanismFluidDisplacement1942}, a simple model of two-phase flow. Given a mixture of water and oil in a tube of porous media, the water fraction $u$ in the model is governed by a scalar conservation law with the non-convex flux function
\begin{equation}
  f(u) = \frac{u^2}{u^2 + \frac 1 2 (1-u)^2}.
\end{equation}
Again, we consider a 2-to-1 network and impose the model on its edges by taking the flux functions $f_1=f_2=f_3=f$. In a numerical experiment we reproduce a scenario, in which water is pumped into two tubes in order to displace oil and enforce its outflow through a third tube. To this end we use the initial data
\begin{equation}
  u^{0,1}(x)=
  \begin{cases}
    1 & \text{if } x\leq -0.5,\\
    0 & \text{if } x> -0.5
  \end{cases},\quad
  u^{0,2} = 0.16, \quad u^{0,3}=0
\end{equation}
and homogeneous Neumann boundary conditions at the edges. The numerical solution computed by the central MUSCL scheme employing $m=300$ cells per edge, $\text{CFL=0.49}$ and $\lambda=2.5$ is shown in Figure \ref{fig:BL}. In the first incoming edge a shock wave  is formed that is followed by a rarefaction wave and passes through the coupling node to the outgoing edge, where it interacts with a second shock wave originating from the second incoming edge. We emphasize that our numerical approach resolved these network dynamics without analysis of the underlying complex (due to the non-convexity of the flux function) Riemann problem.
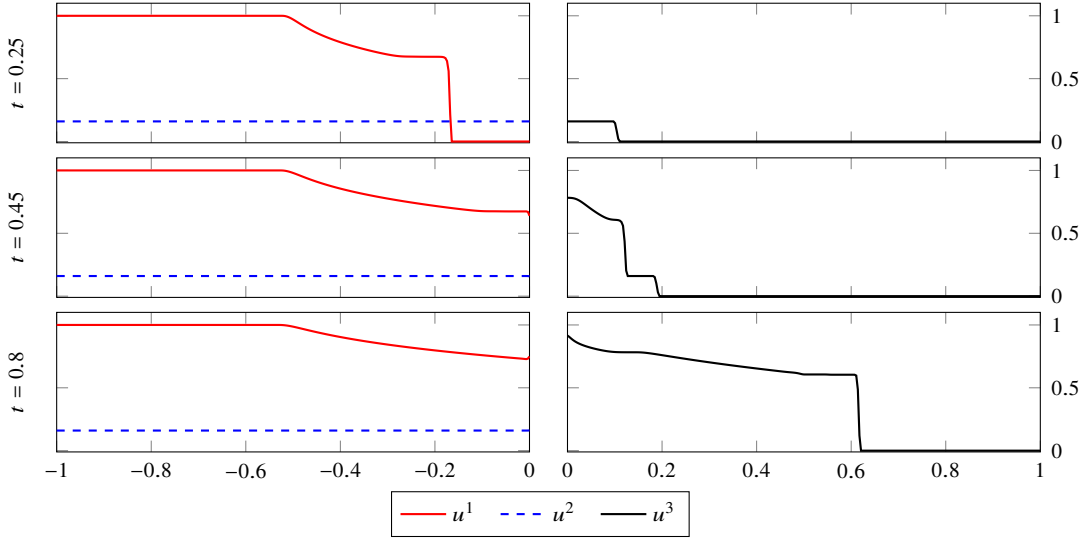
\begin{figure}
  \centering
  \begin{tikzpicture}
  \begin{groupplot}[
        group style={group size=2 by 3,
            horizontal sep = .5 cm, 
            vertical sep = .2 cm,
            xticklabels at=edge bottom,
            yticklabels at=edge right}, 
          width = .5 \linewidth,
          height = .22 \linewidth,
          xmin=-1,xmax=0,ymin=-.01, ymax=1.1,
          every tick label/.append style={font=\scriptsize},
          label style={font=\scriptsize},
          legend to name=leg:BL,
          legend columns=-1,
          legend style={font=\scriptsize, fill=none, /tikz/every even column/.append style={column sep=.25cm}}
          ]
        \nextgroupplot[ylabel={$t=0.25$}]
        \addplot [color=red, thick] table [x index=0, y index=3] {input/BL21_2.dat};
        \addplot [color=blue, thick, dashed] table [x index=0, y index=4] {input/BL21_2.dat};
        \nextgroupplot[xmin=0,xmax=1]
        \addplot [color=black, thick] table [x index=2, y index=5] {input/BL21_2.dat};
        \nextgroupplot[ylabel={$t=0.45$}]
        \addplot [color=red, thick] table [x index=0, y index=3] {input/BL21_3.dat};
        \addplot [color=blue, thick, dashed] table [x index=0, y index=4] {input/BL21_3.dat};
        \nextgroupplot[xmin=0,xmax=1]
        \addplot [color=black, thick] table [x index=2, y index=5] {input/BL21_3.dat};
        \nextgroupplot[ylabel={$t=0.8$}]
        \addplot [color=red, thick] table [x index=0, y index=3] {input/BL21_4.dat};
        \addplot [color=blue, thick, dashed] table [x index=0, y index=4] {input/BL21_4.dat};
        \nextgroupplot[xmin=0,xmax=1]
        \addplot [color=red, thick] coordinates {(-5,0)};
        \addplot [color=blue, thick, dashed] coordinates {(-5,0)};
        \addplot [color=black, thick] table [x index=2, y index=5] {input/BL21_4.dat};
        \legend{$u^1$, $u^2$, $u^3$}
      \end{groupplot}
    \end{tikzpicture}
    \ref{leg:BL}
  \caption{Incoming (left) and outgoing (right) edges in the numerical solution of the Buckley--Leverett equation on a 2-to-1 network. The solution was computed by the central MUSCL scheme using $\text{CFL}=0.49$ and $m=300$ cells per edge.}\label{fig:BL} 
\end{figure}

{\small {\bf Acknowledgment}
The authors thank the Deutsche Forschungsgemeinschaft (DFG, German Research Foundation) for the financial support through 320021702/GRK2326,  333849990/IRTG-2379, CRC1481, HE5386/18-1,19-2,22-1,23-1, ERS SFDdM035 and under Germany’s Excellence Strategy EXC-2023 Internet of Production 390621612 and under the Excellence Strategy of the Federal Government and the Länder. Support through the EU ITN DATAHYKING is also acknowledged. The authors acknowledge the support of the Banff International Research Station (BIRS) for the Focused Research Group [22frg198] “Novel perspectives in kinetic equations for emerging phenomena”, July 17-24, 2022, where part of this work was done. }

\appendix
\section{Eigenvalue analysis of the relaxation system}\label{app:eigenvalues}
In this appendix we derive eigenvalues, eigenvectors and characteristic variables for the relaxation system \eqref{eq:relsyst} with given relaxation rate and speed $\eps, \lambda>0$. We rewrite the system in vector form as
\begin{equation}\label{eq:relsystvector}
  \ddt{\mathbf q^\eps} +
  \begin{pmatrix}
    0 & 1 \\
    \lambda^2 & 0
  \end{pmatrix} \ddx {\mathbf q^\eps} = \frac{1}{\eps}
  \begin{pmatrix}
    0 \\
    f(u^\eps)-v^\eps
  \end{pmatrix}
\end{equation}
using the notation $ \mathbf q^\eps = (u^\eps, q^\eps)$. This is a linear system of balance laws and by diagonalizing the flux matrix we obtain
\begin{equation}
    \begin{pmatrix}
    0 & 1 \\[3pt]
    \lambda^2 & 0
    \end{pmatrix}
    =
    \begin{pmatrix}
    -\frac{1}{\lambda} & \frac{1}{\lambda} \\[3pt]
    1 & 1
    \end{pmatrix}
    \begin{pmatrix}
    -\lambda & 0 \\[3pt]
    0 & \lambda
    \end{pmatrix}
    \begin{pmatrix}
      -\frac{\lambda}{2} & \frac 1 2 \\[3pt]
      \frac \lambda 2 & \frac 1 2
    \end{pmatrix}
    \eqcolon
    R \Lambda R^{-1},
  \end{equation}
  which reveals $-\lambda$ and $\lambda$ as eigenvalues of the system whereas its eigenvectors are given by $\mathbf r_1= (-  1/ \lambda , 1)^T$ and $\mathbf r_2= ( 1/ \lambda , 1)^T$. The characteristic variables consequently read
\begin{equation}
  R^{-1} \mathbf q^\eps = \frac 1 2
  \begin{pmatrix}
    v^\eps - \lambda u^\eps \\
    v^\eps + \lambda u^\eps
  \end{pmatrix}
  \eqcolon
 \begin{pmatrix}
   w^{\eps-}\\
   w^{\eps+}
  \end{pmatrix}.
 \end{equation}

\section{Schemes for differing relaxation speeds}\label{app:speeds}
In this appendix we provide evolution formulas and schemes for differing relaxation speeds in the coupled relaxation system \eqref{eq:relsyst11}. In Section \ref{sec:scheme}, the corresponding formulas are, for brevity, only discussed in the simplified case of equal relaxation speeds left and right from the coupling node (see Remark \ref{rem:equalspeeds}).

\paragraph{The coupled semi-discrete scheme}
Analogously to \eqref{eq:relaxationupwindschemeleftonelambda} and \eqref{eq:relaxationupwindschemerightonelambda} in Section \ref{sec:relscheme11} we obtain by substituting the coupling data \eqref{eq:relaxationcouplingdata} into \eqref{eq:relaxationupwindschemeleft} and \eqref{eq:relaxationupwindschemeright} in the general case 
\begin{subequations}
  \begin{align}
    \ddt u^\eps_{-1} &+ \frac{1}{2 \Delta x} \left( \frac{2 \lambda_2}{\lambda_1 +\lambda_2}v^\eps_0 + \frac{\lambda_1 - \lambda_2}{\lambda_1 + \lambda_2} v^\eps_{-1} - v^\eps_{-2} \right) \notag \\
                          & - \frac{\lambda_1}{2 \Delta x} \left(  \frac{2 \lambda_2^2 }{\lambda_1(\lambda_1 + \lambda_2)}u^\eps_0  - \left[ 2 + \frac{\lambda_1 - \lambda_2}{\lambda_1 + \lambda_2}\right] u^\eps_{-1} + u^\eps_{-2} \right) = 0 \\
    \ddt v^\eps_{-1} &+ \frac{\lambda_1^2}{2 \Delta x} \left(  \frac{2 \lambda_2^2 }{\lambda_1(\lambda_1 + \lambda_2)}u^\eps_N  + \frac{\lambda_2 - \lambda_1}{\lambda_1 + \lambda_2} u^\eps_{-1}- u^\eps_{-2} \right) \notag \\
                          &- \frac{\lambda_1}{2 \Delta x}  \left(  \frac{2 \lambda_2}{\lambda_1 +\lambda_2} v^\eps_0   - \left[2 + \frac{\lambda_2 - \lambda_1}{\lambda_1 + \lambda_2}\right] v^\eps_{-1} + v^\eps_{-2} \right) = \frac{1}{\eps} \left( f_1(u^\eps_{ -1}) - v^\eps_{-1} \right)
  \end{align}
\end{subequations}
for the evolution of the volumes left to the coupling node and 
\begin{subequations}
  \begin{align}
    \ddt u^\eps_{0} &+ \frac{1}{2 \Delta x} \left( v^\eps_{1} -\frac{\lambda_2 - \lambda_1}{\lambda_1 +\lambda_2}v^\eps_0 - \frac{2 \lambda_2}{\lambda_1 + \lambda_2} v^\eps_{-1}   \right) \notag \\
                          & - \frac{\lambda_2}{2 \Delta x} \left(u^\eps_{1}  - \left[ 2 + \frac{\lambda_2 - \lambda_1}{\lambda_1 + \lambda_2}\right] u^\eps_0  + \frac{2 \lambda_1^2 }{\lambda_2(\lambda_1 + \lambda_2)} u^\eps_{-1}\right) = 0, \\
    \ddt v^\eps_{0} &+ \frac{\lambda_2^2}{2 \Delta x} \left( u^\eps_{1} - \frac{\lambda_1 - \lambda_2}{\lambda_1 + \lambda_2}u^\eps_0  -  \frac{2 \lambda_1^2 }{\lambda_2(\lambda_1 + \lambda_2)}u^\eps_{-1} \right) \notag \\
                          &- \frac{\lambda_2}{2 \Delta x}  \left(v^\eps_{1} -   \left[ 2 + \frac{\lambda_1 - \lambda_2}{\lambda_1 + \lambda_2}\right]v^\eps_0   + \frac{2 \lambda_1}{\lambda_1 +\lambda_2} v^\eps_{-1}  \right) = \frac{1}{\eps} \left( f_2(u^\eps_{0}) - v^\eps_{0} \right),
  \end{align}
\end{subequations}
for the evolution of the volumes right to the coupling node. Consistency in the case $f_1=f_2$ to the scheme \eqref{eq:relaxationupwindscheme} in the uncoupled case is only given if also $\lambda_1=\lambda_2$, see Proposition \ref{prop:consistency}.

\paragraph{The limit scheme}
If we allow for differing relaxation speeds when substituting the coupling data given by \eqref{eq:relaxationcouplingdata} and \eqref{eq:limitschemers} into \eqref{eq:limitschemel} and \eqref{eq:limitschemer}, in analogy to the derivation of \eqref{eq:limitschemesub} in Section \ref{sec:limit}, we obtain
\begin{subequations}
  \begin{align}
    u^{n+1}_{-1} &= u^{n}_{-1} - \frac{\Delta t}{2 \Delta x}\left( \frac{2 \lambda_2}{\lambda_1 +\lambda_2}f_2(u_0^n) + \frac{\lambda_1 - \lambda_2}{\lambda_1 + \lambda_2} f_1(u_{-1}^n) - f_1(u_{-2}^n) \right) \notag \\
    &\quad + \frac{\lambda_1}{2 \Delta x} \left(  \frac{2 \lambda_2^2 }{\lambda_1(\lambda_1 + \lambda_2)}u^n_0  - \left[ 2 + \frac{\lambda_1 - \lambda_2}{\lambda_1 + \lambda_2}\right] u^n_{-1} + u^n_{-2} \right), \label{eq:limitschemesubl}\\
    u^{n+1}_{0} &= u^{n}_{0} - \frac{\Delta t}{2 \Delta x}  \left( f_2(u^n_{1}) -\frac{\lambda_2 - \lambda_1}{\lambda_1 +\lambda_2}f_2(u^n_0) - \frac{2 \lambda_2}{\lambda_1 + \lambda_2} f_1(u^n_{-1})   \right) \notag \\
    &\quad + \frac{\lambda_2 \Delta t}{2 \Delta x}  \left(u^n_{1}  - \left[ 2 + \frac{\lambda_2 - \lambda_1}{\lambda_1 + \lambda_2}\right] u^n_0  + \frac{2 \lambda_1^2 }{\lambda_2(\lambda_1 + \lambda_2)} u^n_{-1}\right).\label{eq:limitschemesubr}                                                  
    \end{align}
  \end{subequations}
  These formulas then replace \eqref{eq:limitschemel} and \eqref{eq:limitschemer} in scheme \eqref{eq:limitscheme}. Also in this more general case the limit scheme can be rewritten in the conservative form \eqref{eq:conservativeform} using modified numerical fluxes given by
\begin{equation}\label{eq:numfluxgeneralspeeds}
   F_{j-1/2}^n = \begin{cases}
                 \frac 1 2 \, (f_1(u_{j-1}^n) + f_1(u_{j}^n))  - \frac {\lambda_1} 2 (u_j^n-u_{j-1}^n)  &\text{if }j<0, \\[5pt]
                 \frac 1 {\lambda_1 + \lambda_2} \, ( \lambda_1 f_1(u_{-1}^n) + \lambda_2 f_2(u_{0}^n))  - \frac 1  {\lambda_1 + \lambda_2} ( \lambda_2^2\, u_0^n- \lambda_1^2\, u_{-1}^n) &\text{if }j=0, \\[5pt]
                 \frac 1 2 \, ( f_2(u_{j-1}^n) + f_2(u_{j}^n))  - \frac {\lambda_2} 2 (u_j^n-u_{j-1}^n) &\text{if }j>0. 
               \end{cases}
\end{equation}
\bibliographystyle{abbrv}      
\bibliography{coupling.bib}

\begin{thebibliography}{10}

\bibitem{Banda2016}
M.~K. Banda, A.~Haeck, and M.~Herty.
\newblock Numerical discretization of coupling conditions by high-order
  schemes.
\newblock {\em J. Sci. Comput.}, 2016.

\bibitem{MR2223073}
M.~K. Banda, M.~Herty, and A.~Klar.
\newblock Coupling conditions for gas networks governed by the isothermal
  {E}uler equations.
\newblock {\em Netw. Heterog. Media}, 1(2):295--314, 2006.

\bibitem{MR3265942}
M.~K. Banda, M.~Herty, and J.~M.~T. Ngnotchouye.
\newblock On linearized coupling conditions for a class of isentropic
  multiphase drift-flux models at pipe-to-pipe intersections.
\newblock {\em J. Comput. Appl. Math.}, 276:81--97, 2015.

\bibitem{bezanson2017julia}
J.~Bezanson, A.~Edelman, S.~Karpinski, and V.~B. Shah.
\newblock Julia: A fresh approach to numerical computing.
\newblock {\em SIAM Rev.}, 59(1):65--98, 2017.

\bibitem{MR3528308}
R.~Borsche.
\newblock Numerical schemes for networks of hyperbolic conservation laws.
\newblock {\em Appl. Numer. Math.}, 108:157--170, 2016.

\bibitem{MR3227276}
R.~Borsche and J.~Kall.
\newblock A{DER} schemes and high order coupling on networks of hyperbolic
  conservation laws.
\newblock {\em J. Comput. Phys.}, 273:658--670, 2014.

\bibitem{MR2735916}
B.~Boutin, C.~Chalons, and P.-A. Raviart.
\newblock Existence result for the coupling problem of two scalar conservation
  laws with {R}iemann initial data.
\newblock {\em Math. Models Methods Appl. Sci.}, 20(10):1859--1898, 2010.

\bibitem{MR3200227}
A.~Bressan, S.~\v{C}ani\'{c}, M.~Garavello, M.~Herty, and B.~Piccoli.
\newblock Flows on networks: recent results and perspectives.
\newblock {\em EMS Surv. Math. Sci.}, 1(1):47--111, 2014.

\bibitem{MR2198203}
G.~Bretti, R.~Natalini, and B.~Piccoli.
\newblock Fast algorithms for the approximation of a traffic flow model on
  networks.
\newblock {\em Discrete Contin. Dyn. Syst. Ser. B}, 6(3):427--448, 2006.

\bibitem{MR2818413}
J.~Brouwer, I.~Gasser, and M.~Herty.
\newblock Gas pipeline models revisited: model hierarchies, nonisothermal
  models, and simulations of networks.
\newblock {\em Multiscale Model. Simul.}, 9(2):601--623, 2011.

\bibitem{buckleyMechanismFluidDisplacement1942}
S.~Buckley and M.~Leverett.
\newblock Mechanism of {{Fluid Displacement}} in {{Sands}}.
\newblock {\em Transact. AIME}, 146(01):107--116, Dec. 1942.

\bibitem{MR3315275}
S.~Canic, B.~Piccoli, J.-M. Qiu, and T.~Ren.
\newblock Runge-{K}utta discontinuous {G}alerkin method for traffic flow model
  on networks.
\newblock {\em J. Sci. Comput.}, 63(1):233--255, 2015.

\bibitem{chapmanMathematicalTheoryNonuniform1990}
S.~Chapman and T.~G. Cowling.
\newblock {\em The Mathematical Theory of Non-Uniform Gases: An Account of the
  Kinetic Theory of Viscosity, Thermal Conduction, and Diffusion in Gases}.
\newblock Cambridge Mathematical Library. {Cambridge University Press},
  {Cambridge ; New York}, 3rd ed edition, 1990.

\bibitem{chenHyperbolicConservationLaws1994}
G.-Q. Chen, C.~D. Levermore, and T.-P. Liu.
\newblock Hyperbolic conservation laws with stiff relaxation terms and entropy.
\newblock {\em Comm. Pure Appl. Math.}, 47(6):787--830, June 1994.

\bibitem{MR2377285}
R.~M. Colombo and M.~Garavello.
\newblock On the {C}auchy problem for the {$p$}-system at a junction.
\newblock {\em SIAM J. Math. Anal.}, 39(5):1456--1471, 2008.

\bibitem{MR2438778}
R.~M. Colombo, M.~Herty, and V.~Sachers.
\newblock On {$2\times2$} conservation laws at a junction.
\newblock {\em SIAM J. Math. Anal.}, 40(2):605--622, 2008.

\bibitem{crandallMethodFractionalSteps1980}
M.~Crandall and A.~Majda.
\newblock The method of fractional steps for conservation laws.
\newblock {\em Numer. Math.}, 34(3):285--314, Sept. 1980.

\bibitem{MR2665143}
C.~D'Apice, S.~G\"{o}ttlich, M.~Herty, and B.~Piccoli.
\newblock {\em Modeling, simulation, and optimization of supply chains}.
\newblock Society for Industrial and Applied Mathematics (SIAM), Philadelphia,
  PA, 2010.
\newblock A continuous approach.

\bibitem{duboisBoundaryConditionsNonlinear1988}
F.~Dubois and P.~Le~Floch.
\newblock Boundary conditions for nonlinear hyperbolic systems of conservation
  laws.
\newblock {\em J. Differ. Equations}, 71(1):93--122, Jan. 1988.

\bibitem{MR3744998}
H.~Egger.
\newblock A robust conservative mixed finite element method for isentropic
  compressible flow on pipe networks.
\newblock {\em SIAM J. Sci. Comput.}, 40(1):A108--A129, 2018.

\bibitem{formaggiaMultiscaleModellingCirculatory1999}
L.~Formaggia, F.~Nobile, A.~Quarteroni, and A.~Veneziani.
\newblock Multiscale modelling of the circulatory system: A preliminary
  analysis.
\newblock {\em Comput Visual Sci}, 2(2-3):75--83, Dec. 1999.

\bibitem{MR3553143}
M.~Garavello, K.~Han, and B.~Piccoli.
\newblock {\em Models for vehicular traffic on networks}, volume~9 of {\em AIMS
  Series on Applied Mathematics}.
\newblock American Institute of Mathematical Sciences (AIMS), Springfield, MO,
  2016.

\bibitem{garavelloTrafficFlowNetworks2006}
M.~Garavello and B.~Piccoli.
\newblock {\em Traffic Flow on Networks: Conservation Law Models}.
\newblock Number Vol. 1 in {{AIMS}} Series on Applied Mathematics. {American
  Inst. of Mathematical Sciences}, {Springfield, Mo}, 2006.

\bibitem{godlewskiNumericalApproximationHyperbolic1996}
E.~Godlewski and P.-A. Raviart.
\newblock {\em Numerical {{Approximation}} of {{Hyperbolic Systems}} of
  {{Conservation Laws}}}, volume 118 of {\em Applied {{Mathematical
  Sciences}}}.
\newblock {Springer New York}, {New York, NY}, 1996.

\bibitem{MR2045460}
E.~Godlewski and P.-A. Raviart.
\newblock The numerical interface coupling of nonlinear hyperbolic systems of
  conservation laws. {I}. {T}he scalar case.
\newblock {\em Numer. Math.}, 97(1):81--130, 2004.

\bibitem{gottlichSecondOrderTrafficFlow2021}
S.~G{\"o}ttlich, M.~Herty, S.~Moutari, and J.~Weissen.
\newblock Second-{{Order Traffic Flow Models}} on {{Networks}}.
\newblock {\em SIAM J. Appl. Math.}, 81(1):258--281, Jan. 2021.

\bibitem{GugatHertyMueller:2017}
M.~Gugat, M.~Herty, and S.~M{\"u}ller.
\newblock Coupling conditions for the transition from supersonic to subsonic
  fluid states.
\newblock {\em Netw. Heterog. Media}, 12(3):371--380, 2017.

\bibitem{HantkeMueller:2018}
M.~Hantke and S.~M{\"u}ller.
\newblock Analysis and simulation of a new multi-component two-phase flow model
  with phase transitions and chemical reactions.
\newblock {\em Quart. Appl. Math.}, 76(2):253--287, Jan. 2018.

\bibitem{HantkeMueller:2019}
M.~Hantke and S.~M{\"u}ller.
\newblock Closure conditions for a one temperature non-equilibrium
  multi-component model of baer-nunziato type.
\newblock {\em ESAIM: ProcS}, 66:42--60, 2019.

\bibitem{HertyMuellerGerhardXiangWang:2018}
M.~Herty, S.~M{\"u}ller, N.~Gerhard, G.~Xiang, and B.~Wang.
\newblock Fluid-structure coupling of linear elastic model with compressible
  flow models: {{Coupling}} of linear elastic model with compressible flow
  models.
\newblock {\em Int. J. Numer. Meth. Fluids}, 86(6):365--391, Feb. 2018.

\bibitem{MR2237163}
M.~Herty and M.~Rascle.
\newblock Coupling conditions for a class of second-order models for traffic
  flow.
\newblock {\em SIAM J. Math. Anal.}, 38(2):595--616, 2006.

\bibitem{MR1338371}
H.~Holden and N.~H. Risebro.
\newblock A mathematical model of traffic flow on a network of unidirectional
  roads.
\newblock {\em SIAM J. Math. Anal.}, 26(4):999--1017, 1995.

\bibitem{MR4175145}
Y.~Holle, M.~Herty, and M.~Westdickenberg.
\newblock New coupling conditions for isentropic flow on networks.
\newblock {\em Netw. Heterog. Media}, 15(4):605--631, 2020.

\bibitem{huAsymptoticPreservingSchemesMultiscale2017}
J.~Hu, S.~Jin, and Q.~Li.
\newblock Asymptotic-{{Preserving Schemes}} for {{Multiscale Hyperbolic}} and
  {{Kinetic Equations}}.
\newblock In {\em Handbook of {{Numerical Analysis}}}, volume~18, pages
  103--129. {Elsevier}, 2017.

\bibitem{jin2010asymptotic}
S.~Jin.
\newblock Asymptotic preserving ({{AP}}) schemes for multiscale kinetic and
  hyperbolic equations: A review.
\newblock In {\em Lecture Notes for Summer School on Methods and Models of
  Kinetic Theory ({{M}}\&{{MKT}})}, pages 177--216. {Porto Ercole (Grosseto,
  Italy)}.

\bibitem{jinRelaxationScheme1995}
S.~Jin and Z.~Xin.
\newblock The relaxation schemes for systems of conservation laws in arbitrary
  space dimensions.
\newblock {\em Commun. Pure Appl. Math.}, 48(3):235--276, 1995.

\bibitem{k.bandaCouplingConditionsGas2006}
M.~K.~Banda, M.~Herty, and A.~Klar.
\newblock Coupling conditions for gas networks governed by the isothermal
  {{Euler}} equations.
\newblock {\em Netw. Heterog. Media}, 1(2):295--314, 2006.

\bibitem{KarlsenKlingenbergRisebro:2004}
K.~H. Karlsen, C.~Klingenberg, and N.~H. Risebro.
\newblock A {Relaxation} {Scheme} for {Conservation} {Laws} with a
  {Discontinuous} {Coefficient}.
\newblock {\em Math. Comp.}, 73(247):1235--1260, Dec. 2003.

\bibitem{karlsenConvergenceGodunovScheme2017}
K.~H. Karlsen and J.~D. Towers.
\newblock Convergence of a {{Godunov}} scheme for conservation laws with a
  discontinuous flux lacking the crossing condition.
\newblock {\em J. Hyper. Differential Equations}, 14(04):671--701, Dec. 2017.

\bibitem{MR2600931}
O.~Kolb, J.~Lang, and P.~Bales.
\newblock An implicit box scheme for subsonic compressible flow with
  dissipative source term.
\newblock {\em Numer. Algorithms}, 53(2-3):293--307, 2010.

\bibitem{CentralNetworkSchemeCode}
N.~Kolbe.
\newblock Implementation of central schemes for networks of scalar conservation
  laws.
\newblock GitHub repository,
  \url{https://github.com/nklb/CentralNetworkScheme}, 2022.

\bibitem{kurganovSemidiscreteCentralupwindSchemes2001}
A.~Kurganov, S.~Noelle, and G.~Petrova.
\newblock Semidiscrete central-upwind schemes for hyperbolic conservation laws
  and {{Hamilton-Jacobi}} equations.
\newblock {\em SIAM J. Sci. Comput.}, 23(3):707--740, 2001.

\bibitem{levequeFiniteVolumeMethods2002}
R.~J. LeVeque.
\newblock {\em Finite Volume Methods for Hyperbolic Problems}.
\newblock Cambridge {{Texts}} in {{Applied Mathematics}}. {Cambridge University
  Press, Cambridge}, 2002.

\bibitem{lighthillKinematicWavesII1955}
M.~J. Lighthill and G.~B. Whitham.
\newblock On kinematic waves {{II}}. {{A}} theory of traffic flow on long
  crowded roads.
\newblock {\em Proc. R. Soc. Lond. A}, 229(1178):317--345, May 1955.

\bibitem{liuHyperbolicConservationLaws1987}
T.-P. Liu.
\newblock Hyperbolic conservation laws with relaxation.
\newblock {\em Commun. Math. Phys.}, 108(1):153--175, Mar. 1987.

\bibitem{MR4026004}
Y.~Mantri, M.~Herty, and S.~Noelle.
\newblock Well-balanced scheme for gas-flow in pipeline networks.
\newblock {\em Netw. Heterog. Media}, 14(4):659--676, 2019.

\bibitem{MR4039520}
P.~Mindt, J.~Lang, and P.~Domschke.
\newblock Entropy-preserving coupling of hierarchical gas models.
\newblock {\em SIAM J. Math. Anal.}, 51(6):4754--4775, 2019.

\bibitem{MR3396266}
L.~O. M\"{u}ller and P.~J. Blanco.
\newblock A high order approximation of hyperbolic conservation laws in
  networks: application to one-dimensional blood flow.
\newblock {\em J. Comput. Phys.}, 300:423--437, 2015.

\bibitem{MuellerVoss:2006}
S.~M{\"u}ller and A.~Voss.
\newblock The {{Riemann Problem}} for the {{Euler Equations}} with
  {{Nonconvex}} and {{Nonsmooth Equation}} of {{State}}: {{Construction}} of
  {{Wave Curves}}.
\newblock {\em SIAM J. Sci. Comput.}, 28(2):651--681, Jan. 2006.

\bibitem{richardsShockWavesHighway1956}
P.~I. Richards.
\newblock Shock {{Waves}} on the {{Highway}}.
\newblock {\em Oper. Res.}, 4(1):42--51, Feb. 1956.

\bibitem{MR2963941}
S.~Tan and C.-W. Shu.
\newblock Inverse {L}ax-{W}endroff procedure for numerical boundary conditions
  of hyperbolic equations: survey and new developments.
\newblock In {\em Advances in applied mathematics, modeling, and computational
  science}, volume~66 of {\em Fields Inst. Commun.}, pages 41--63. Springer,
  New York, 2013.

\bibitem{towersExplicitFiniteVolume2022}
J.~D. Towers.
\newblock An explicit finite volume algorithm for vanishing viscosity solutions
  on a network.
\newblock {\em Netw. Heterog. Media}, 17(1):1, 2022.

\bibitem{vanleerUltimateConservativeDifference1979}
B.~{van Leer}.
\newblock Towards the ultimate conservative difference scheme. {{V}}. {{A}}
  second-order sequel to {{Godunov}}'s method.
\newblock {\em J. Comput. Phys.}, 32(1):101--136, July 1979.

\bibitem{MR4260434}
X.~Wu and J.~Chan.
\newblock Entropy stable discontinuous {G}alerkin methods for nonlinear
  conservation laws on networks and multi-dimensional domains.
\newblock {\em J. Sci. Comput.}, 87(3):Paper No. 100, 34, 2021.

\end{thebibliography}

\end{document}